\documentclass[11pt]{article}
\usepackage[utf8]{inputenc}
\usepackage{geometry}
\usepackage{xcolor}
\usepackage{graphicx} 
\usepackage{amsthm}   
\usepackage{mathtools}
\usepackage{amsmath}  
\usepackage{amsfonts} 
\usepackage{setspace} 
\usepackage{cancel}   
\usepackage{amssymb}  
\usepackage{enumitem} 
\usepackage{float}    
\usepackage{mathrsfs} 
\usepackage{hyperref} 
\usepackage{array}
\usepackage{nccmath}  
\hypersetup{colorlinks,linkcolor={blue},citecolor={red},urlcolor={red}} 

\usepackage{rotating}
\usepackage{cleveref} 
\usepackage{stackrel} 
\usepackage{subfig} 
\usepackage{multirow} 

\newtheorem{theorem}{Theorem}[section]

\newtheorem{corollary}[theorem] {Corollary}
\newtheorem{lemma} [theorem]{Lemma}
\newtheorem{definition}{Definition}[section]

\theoremstyle{remark}
\newtheorem{remark}[theorem]{Remark}
\newtheorem{example}[theorem]{Example}
\newcommand{\PreserveBackslash}[1]{\let\temp=\\#1\let\\=\temp}
\newcolumntype{C}[1]{>{\PreserveBackslash\centering}p{#1}}
\newcolumntype{R}[1]{>{\PreserveBackslash\raggedleft}p{#1}}
\newcolumntype{L}[1]{>{\PreserveBackslash\raggedright}p{#1}}

\newcommand{\N}{\mathbb{N}}

\newcommand{\I}{\mathcal{K}}

\title{Decay of Chebyshev coefficients and error estimates of associated quadrature on shrinking intervals}
\author{Krishna Yamanappa Poojara \footnote{Computing and Mathematical Sciences, Caltech, Pasadena, CA 91125, USA, kyp6174@caltech.edu}, \and Sabhrant Sachan\footnote{Computing and Mathematical Sciences, Caltech, Pasadena, CA 91125, USA, ssachan@caltech.edu}, \and Ambuj Pandey \footnote{Indian Institute of Science Education and Research Bhopal (IISER Bhopal), ambuj@iiserb.ac.in}}
\date{}

\begin{document}
	
	\maketitle
    	\begin{center}
		Preliminary draft
	\end{center}
	\begin{abstract}
		We analyze decay of Chebyshev coefficients and local Chebyshev approximations for functions of finite regularity on finite intervals, focusing on the framework where the interval length tends to zero while the number of approximation nodes remains fixed. For all four families of Chebyshev polynomials, we derive error estimates that quantify the dependence on the interval length for (i) the decay of Chebyshev coefficients, (ii) the approximation error between continuous and discrete Chebyshev coefficients, and (iii) the convergence of Chebyshev-based quadrature rules.  These results fill a gap in the existing theory and provide a unified and rigorous description of how approximation accuracy scales on shrinking intervals, offering new theoretical insight and practical guidance for high-order numerical methods on decomposed domains. Numerical experiments corroborate the theoretical results, confirming the decay rates and illustrating the error behavior in practice.
	\end{abstract}
	
	\section{Introduction}

    Among the various families of orthogonal polynomials, the four classical Chebyshev families (of the first, second, third, and fourth kinds) play a central role due to their wide-ranging applications in interpolation, spectral methods, quadrature, and the numerical solution of differential and integral equations.  They are widely used, with the first-kind family distinguished by its near-minimal growth~\cite{arnold2001concise}, and all four families valued for their stable numerical behavior and compatibility with efficient algorithms such as the Fast Fourier Transform (FFT) see e.g. \cite{trefethen2019approximation,boyd2001chebyshev,mason2002chebyshev} and non-uniform FFT \cite{liu2000applications}. 
	Classical results on Chebyshev expansions primarily concern global approximation on the standard interval $[-1,1]$ \cite{xiang2010error,xiang2012convergence}. However, many modern applications---including spectral methods for differential equations, quadrature on non-standard domains, and localized wave scattering computations---require accurate local approximations on finite and often shrinking intervals $[a,b]\subset \mathbb{R}$. In particular, local Chebyshev expansions on finite subintervals have become central to the design of high-order numerical methods for singular integrals, fast multipole methods, and adaptive algorithms where the computational domain is decomposed into smaller subdomains, see e.g. \cite{bruno2020chebyshev,hu2021chebyshev,hao2014high,sideris2025high} and references therein. A key feature of these methods is that computational efficiency is maintained by fixing the number of approximation nodes, while accuracy is improved by shrinking the local approximation interval.

    Despite the central role of Chebyshev polynomials in numerical approximations on shrinking intervals, precise error bounds that explicitly capture the dependence on the interval length for the decay of Chebyshev coefficients and the error in Fej\'er quadratures are not systematically available in the literature. Existing results  over general finite intervals provide useful estimates, but they do not describe in detail how the approximation error depends on $b-a$, particularly in the asymptotic regime where $b-a \to 0$. This omission leaves a gap in the theory precisely in the context most relevant for localized high-order methods. To be precise, when employing local Chebyshev expansions on shrinking subintervals, two natural questions arise:
	\begin{enumerate}[label=\roman*.]
		\item How does the decay of Chebyshev coefficients depend on the regularity of the underlying function and on the interval length?
        
		\item How accurately can continuous Chebyshev coefficients be approximated by their discrete counterparts, particularly when the number of interpolation nodes is fixed but the interval length tends to zero?
	\end{enumerate}
	These questions are of both theoretical and practical importance. From a theoretical standpoint, precise decay bounds clarify the asymptotic behavior of Chebyshev expansions under local scaling on shrinking domains. From a practical standpoint, such results guide the design of efficient algorithms in which the number of Chebyshev nodes per subinterval remains fixed, while accuracy is enhanced by subdividing the computational domain into smaller pieces.
	
	In this paper, we develop error estimates that explicitly describe how the decay of Chebyshev coefficients, the error in their approximation by discrete counterparts, and the convergence of Chebyshev-based quadrature rules (see Table \ref{T:quadratures}) depend on the interval length, for all four classical families of Chebyshev polynomials (given in \eqref{eq:ChebyPoly}). Our main contributions are as follows:
	\begin{enumerate}
		\item \textbf{Decay bounds for Chebyshev coefficients.}  
		We establish sharp decay estimates for the Chebyshev coefficients of a function of finite regularity $f$ defined on $[a,b]$, applicable to all four Chebyshev families. These bounds explicitly capture the dependence on both the regularity of $f$ and the interval length $b-a$, thereby quantifying the rate at which the $k$-th coefficient decays as $b-a \to 0$.
        
		\item \textbf{Error between continuous and discrete Chebyshev coefficients.}  
		We derive error bounds for the approximation of continuous Chebyshev coefficients by their discrete counterparts, again for all four kinds of Chebyshev polynomials. The results show how the discrete coefficients converge to the continuous ones when the number of terms is fixed but the approximation interval shrinks ($b-a \to 0$).  
		
		\item \textbf{Applications to Chebyshev-based quadrature rules.}  
        Based on the novel results presented in 1. and 2., we derive error estimates for quadrature rules associated with Chebyshev polynomials of the first, second, third, and fourth kinds. In particular, we characterize the behavior of the quadrature error when the number of quadrature nodes is fixed while the length of the integration interval tends to zero. This analysis offers new insight into the local accuracy and convergence properties of Chebyshev-based quadrature schemes.

	\end{enumerate}

    The paper is organized as follows. Section~\ref{sec:prelim} introduces the four classical families of Chebyshev polynomials and their coefficients, along with interpolatory quadrature rules based on their zeros or extrema; the corresponding discrete coefficients are summarized in Table~\ref{T:quadratures}. Section~\ref{sec:decay} investigates the decay of Chebyshev coefficients, highlighting the effects of function regularity and interval length. In Section~\ref{sec:appln}, we discuss applications of these results, including errors in approximating continuous coefficients by discrete ones and quadrature errors as the interval shrinks. Section~\ref{sec:numerics} presents numerical results that confirm the predicted decay rates and quadrature convergence. We conclude with final remarks in Section~\ref{conclusion}.

	\section{Preliminaries}\label{sec:prelim}
	
	Among the most important classical orthogonal polynomials are the Chebyshev polynomials, which occur in four families: $T_n$, $U_n$, $V_n$, and $W_n$, corresponding to the first, second, third, and fourth kinds, respectively. Explicitly, if $x = \cos\theta$ with $\theta\in[0,\pi]$ and $n\ge 0$, they admit the trigonometric representations \cite{mason2002chebyshev}
	\begin{equation}\label{eq:ChebyPoly}
		T_n(x) = \cos(n\theta), \quad 
		U_n(x) = \frac{\sin((n+1)\theta)}{\sin\theta}, \quad 
		V_n(x) = \frac{\cos((n+1)\theta/2)}{\cos(\theta/2)}, \quad 
		W_n(x) = \frac{\sin((n+1)\theta/2)}{\sin(\theta/2)}.
	\end{equation}
	Each family is orthogonal with respect to a specific Jacobi weight function, and more generally, they can all be viewed as special cases of the Jacobi polynomials, a broad class of orthogonal polynomials. These Jacobi polynomials arise as solutions of the Sturm--Liouville problem \cite[Theorem~4.2.1]{szeg1939orthogonal}
	\begin{equation*}
		(1-x^2)y'' + \bigl(\beta-\alpha - (\alpha+\beta+2)x\bigr)y' + n(n+\alpha+\beta+1)y = 0,
	\end{equation*}
	where $x \in [-1,1]$, the parameters $\alpha,\beta > -1$ and a Jacobi polynomial of degree $n$ is denoted by $P_n^{(\alpha,\beta)}(x)$. They are orthogonal with respect to the weight function
	\[
	w^{\alpha,\beta}(x) = (1-x)^{\alpha}(1+x)^{\beta}.
	\]
	The Chebyshev polynomials of the first, second, third, and fourth kinds appear as a special cases of Jacobi polynomials corresponding to the parameter choices  
	\[
	(\alpha,\beta) = \Bigl(-\tfrac{1}{2},-\tfrac{1}{2}\Bigr), \quad
	(\alpha,\beta) = \Bigl(\tfrac{1}{2},\tfrac{1}{2}\Bigr), \quad
	(\alpha,\beta) = \Bigl(-\tfrac{1}{2},\tfrac{1}{2}\Bigr), \quad \text{and} \quad
	(\alpha,\beta) = \Bigl(\tfrac{1}{2},-\tfrac{1}{2}\Bigr),
	\]
	respectively. In particular, letting 
	\begin{equation}\label{eq:gamma-n}
		\gamma_n = \begin{cases} 1, &\text{if } n =0\\ 2, &\text{if } n\geq 1, \end{cases} \quad \text{and}\quad \delta_{nm} = \begin{cases}
			1, &\text{if } n = m\\ 0, &\text{if } n\neq m, \end{cases}
	\end{equation}
	the orthogonality relations of the Chebyshev polynomials mentioned above are  
	\begin{align}
		&\int_{-1}^1 T_n(x)T_m(x)(1-x^2)^{-1/2}\,dx = \frac{\pi}{\gamma_n}\,\delta_{nm},\\[6pt]
		&\int_{-1}^1 U_n(x)U_m(x)(1-x^2)^{1/2}\,dx = \tfrac{\pi}{2}\,\delta_{nm}, \\[6pt]
		&\int_{-1}^1 V_n(x)V_m(x)(1-x)^{-1/2}(1+x)^{1/2}\,dx = \pi\,\delta_{nm}, \\[6pt]
		&\int_{-1}^1 W_n(x)W_m(x)(1-x)^{1/2}(1+x)^{-1/2}\,dx = \pi\,\delta_{nm}.
	\end{align}
	
	Given a continuous function $f$ on $[-1,1]$, one can expand $f$ as an infinite series in terms of the orthogonal family of Jacobi polynomials:
	\begin{equation}
		f(x) = \sum_{k=0}^{\infty} c_k P_{k}^{(\alpha,\beta)}(x), \quad x \in [-1,1].
	\end{equation}
	
	When the Jacobi polynomials $P_k^{(\alpha,\beta)}$ specialize to the Chebyshev polynomials $T_k$, $U_k$, $V_k$, and $W_k$, we denote the corresponding expansion coefficients by $c_k^{(d)}$ with $d=1,2,3,4$, respectively. These are referred to as the continuous Chebyshev coefficients of the first, second, third, and fourth kinds, respectively. Explicitly, they are defined by  
	\begin{align}
		\label{eq:ck1-cont}
		c_k^{(1)} &= \frac{\gamma_k}{\pi} \int_{-1}^{1} f(x)\,T_k(x)(1-x^2)^{-1/2}\,dx, \\
		c_k^{(2)} &= \frac{2}{\pi} \int_{-1}^{1} f(x)\,U_k(x)(1-x^2)^{1/2}\,dx, \\
		c_k^{(3)} &= \frac{1}{\pi} \int_{-1}^{1} f(x)\,V_k(x)(1-x)^{-1/2}(1+x)^{1/2}\,dx, \\
		c_k^{(4)} &= \frac{1}{\pi} \int_{-1}^{1} f(x)\,W_k(x)(1-x)^{1/2}(1+x)^{-1/2}\,dx,\label{eq:ck4-cont}
	\end{align}
	respectively.
	The decay rates of these coefficients play a crucial role in analyzing the errors of quadrature rules derived from Jacobi polynomials, such as Fejér first and Clenshaw–Curtis quadratures~\cite{xiang2012convergence}.
     The Fej\'er first and second quadratures can be viewed as interpolatory quadratures \cite[Sec.~2.5.5]{davis2007methods} corresponding to the zeros of the Chebyshev polynomials $T_n$ and $U_n$, respectively, while the Clenshaw--Curtis quadrature is an interpolatory quadrature associated with the extrema of $T_n$ \cite[Sec.~6.4.2]{plonka2018numerical}. In a similar manner, we introduce quadratures based on the zeros of $V_n$ and $W_n$. Although these quadratures are known in the literature \cite{notaris1997interpolatory}, no standard naming convention exists. For consistency with the Fej\'er first and second quadratures, we shall refer to them as the ``Fej\'er third'' and ``Fej\'er fourth'' quadratures throughout this paper. In what follows, we briefly summarize these five quadratures.
	\begin{definition}\label{def:quad}
		Quadrature rules based on Chebyshev polynomials approximate the integral of a continuous function $f \in C[a,b]$ by  
		\begin{equation}\label{eq:quad}
			\int_{a}^b f(t)\,dt \;\approx\; Q_n[f],\quad \text{where}\quad Q_n[f] = \frac{h}{2}\sum_{j=0}^{n-1} w_j f(\xi(t_j)), 
			\quad 
			w_j = \int_{-1}^1 L_j(t)\,dt
		\end{equation}
		and $\xi : [-1,1] \to [a,b]$ is defined by $\xi(t) = a+\frac{h}{2}(t+1)$ with $h=b-a$. Here $L_j$ denotes the $j$-th Lagrange basis polynomial associated with the nodes $t_j = \cos(\theta_j)$.  
		We consider the five classical Chebyshev-based interpolatory quadratures: Fejér first (F-I), Fejér second (F-II), Fejér third (F-III), Fejér fourth (F-IV) for $n\ge 1$, and Clensha-Curtis (CC) for $n\ge 2$. Their nodes, weights, basis polynomials, and discrete coefficients are summarized in Table~\ref{T:quadratures}.  
		Here $\gamma$ is defined in \eqref{eq:gamma-n}, and  
		\begin{equation}\label{eq:gamma-tilde}
			\widetilde{\gamma}_j = 
			\begin{cases}
				2, & j=0 \text{ or } j=n-1,\\
				1, & 0 < j < n-1,
			\end{cases}
		\end{equation}
		with `Quad’ denoting the specific quadrature, $\theta_j\in[0,\pi]$ and  $t_j=\cos(\theta_j)\in[-1, 1]$ the quadrature nodes, $w_j$ the corresponding weights, $L_j$ the associated Lagrange basis functions, and $\widetilde{c}_k$ the discrete Chebyshev coefficients defined with respect to these nodes.
		\begin{table}[!htb]
			\centering
			\scalebox{0.88}{
				\begin{tabular}{|C{0.82cm}|C{1.2cm}|C{5.75cm}|C{4.95cm}|C{4.8cm}|}
					\hline
					Quad &  $\theta_j$& $w_j$ & $L_j$ & $\widetilde{c}_k$\\ \hline
					F-I&  $\dfrac{2j+1}{2n}\pi$&  $\displaystyle\frac{2}{n}\left[1-2\sum\limits_{k=1}^{\lfloor \tfrac{n}{2}\rfloor} \frac{\cos\left(2k\theta_j\right)}{4 k^{2}-1} \right]$ & $\displaystyle\frac{1}{n}\sum_{k=0}^{n-1}\gamma_k T_{k}(t_{j})T_{k}(t)$ & $\displaystyle\frac{\gamma_{k}}{n}\sum\limits_{j=0}^{n-1} f(t_{j}) T_{k} (t_j)$\\[4ex] \hline
					CC&  $\dfrac{j\pi}{n-1}$&  $\displaystyle\frac{2}{(n-1)\widetilde{\gamma}_j}\hspace{-2mm}\left[1-\sum_{k=1}^{\lfloor \tfrac{n-1}{2} \rfloor} \frac{2}{\widetilde{\gamma}_{2k}} \frac{\cos{(2k\theta_j)}}{4k^2-1} \right]$ & $\displaystyle\frac{2}{(n-1)\widetilde \gamma_j} \sum_{k=0}^{n-1} \frac{1}{\widetilde \gamma_k} T_k(t_j)T_k(t)$ & $\displaystyle\frac{2}{(n-1)\widetilde \gamma_k} \sum_{j=0}^{n-1} \frac{1}{\widetilde \gamma_j}f(t_j)T_k(t_j)$ \\[4ex] \hline
					F-II&  $\dfrac{j+1}{n+1}\pi$&   $ \displaystyle\frac{4\sin{(\theta_j)}}{n+1}\sum_{k=1}^{\lfloor \tfrac{n+1}{2} \rfloor}  \frac{\sin{((2k-1)\theta_j)}}{2k-1}$ & $\displaystyle\frac{2}{n+1} \sum_{k=0}^{n-1} U_k(t_j)U_k(t) (1-t_j^2)$ & $\displaystyle\frac{2}{n+1}\sum_{j=0}^{n-1}f(t_j)U_k(t_j) (1-t_j^2)$\\[4ex] \hline
					F-III&  $\dfrac{2j+1}{2n+1}\pi$&   $\displaystyle\frac{4\sin{(\theta_j)}}{n+\tfrac{1}{2}}\sum_{k=1}^{\lfloor \tfrac{n+1}{2} \rfloor} \frac{\sin{((2k-1)\theta_j)}}{2k-1}$ & $\displaystyle\frac{1}{n+\tfrac{1}{2}} \sum_{k=0}^{n-1} V_k(t_j)V_k(t) (1+t_j)$ & $\displaystyle\frac{2}{n+\tfrac{1}{2}}\sum_{j=0}^{n-1}f(t_j)V_k(t_j) (1+t_j)$\\[4ex] \hline
					F-IV& $\dfrac{2j+2}{2n+1}\pi$& $\displaystyle\frac{4\sin{(\theta_j)}}{n+\tfrac{1}{2}}\sum_{k=1}^{\lfloor \tfrac{n+1}{2} \rfloor} \frac{\sin{((2k-1)\theta_j)}}{2k-1}$ & $\displaystyle \frac{1}{n+\tfrac{1}{2}} \sum_{k=0}^{n-1} W_k(t_j)W_k(t) (1-t_j)$ & $\displaystyle\frac{2}{n+\tfrac{1}{2}}\sum_{j=0}^{n-1}f(t_j)W_k(t_j) (1-t_j)$\\[4ex] \hline
			\end{tabular}}
			\caption{Nodes are given by $t_j = \cos(\theta_j)$ for $j=0,\dots,n-1$ in each quadrature.}
			\label{T:quadratures}
		\end{table}
	\end{definition}
	
	All five quadratures are interpolatory corresponding to the grids $t_j=\cos(\theta_j)$ with $\theta_j$ are as mentioned in Table~\ref{T:quadratures}. For a given continuous function $f$ in $[-1,1]$, the Lagrange interpolating polynomial is  
	\begin{equation}\label{eq:Lag-poly}
		P_n[f](t) = \sum_{j=0}^{n-1} L_j(t) f(t_j),
	\end{equation}
	with $L_j$ as listed in Table~\ref{T:quadratures}. Moreover, $P_n[f](t)$ can be expressed as a discrete Chebyshev expansion. For example, in the case of F-I,  
	\begin{equation}
		P_n[f](t) = \sum_{k=0}^{n-1} \widetilde{c}_k T_k(t),
	\end{equation}
	where $\widetilde{c}_k$ are the discrete Chebyshev coefficients corresponding to the F-I quadrature grids as given in Table~\ref{T:quadratures}.  
	
	In the next section, we present decay results for the continuous Chebyshev coefficients of functions $f \in X^m([a,b])$ (or simply $ X^m[a,b]$) defined in \eqref{eq:ck1-cont}, where $X^m[a,b]$ denotes the class of $m$-times continuously differentiable functions whose $(m+2)$-th derivative is piecewise continuous, i.e.  
	\begin{equation}\label{eq:Xm-def}
		X^m[a,b]  = C^{m}[a,b] \cap C^{m+2}_{\mathrm{pw}}[a,b].
	\end{equation}
	Throughout this paper, the symbol $C$ stands for a positive constant taking on different values on different occurrences.

	\section{Decay of the Chebyshev Coefficients With Interval Length}\label{sec:decay}
	In this section, we derive a decay estimate of the Chebyshev coefficients in interval $[a,b]$ that not only depends on the regularity of the function $f$ but also on the length of the underlying interval as shrinks to zero. In literature, many authors have discussed the convergence of Chebyshev coefficients, for instance see  \cite{xiang2015error,xiang2010error,xiang2013convergence,trefethen2008gauss} and references therein, however, in best of our knowledge, the decay rate in terms of the length of the interval are not systematically reported. We now examine the decay rates of the coefficients with respect to the interval size $h=b-a$ and the index $k$.
    
    Given a continuous function $f$ on $[a, b]$, for simplicity of notation we write $c_k$ in place of $c_k^{(1)}$ (defined in~\eqref{eq:ck1-cont}), to denote the continuous Chebyshev coefficients of $f$ with respect to $T_n$, where  
\begin{equation}\label{eq:util_def}
    \widetilde{f}(t) = f\bigl(\xi(t)\bigr), 
    \quad \xi(t) = \tfrac{h}{2}t + \tfrac{a+b}{2}, 
    \quad -1 \le t \le 1.
\end{equation}
To estimate $c_k$ in terms of the interval length $h$ and the coefficient index $k$, it is essential to understand the asymptotic behavior of $c_k$ as $h \to 0$. By Lemma~\ref{lem:decay}, we obtain  
\begin{equation*}
    c_0 \;\to\; f\left(\frac{a+b}{2}\right), 
    \quad \text{and} \quad 
    c_k \;\to\; 0 \quad \text{for all} \quad k \ge 1 
    \quad \text{as} \quad h \to 0.
\end{equation*}
In the following, we present the main result of this paper, which establishes the decay rate of the Chebyshev coefficients $c_k^{(\ell)}$ for $k \ge 1$, $\ell=1,2,3,4$ (defined in \eqref{eq:ck1-cont}--\eqref{eq:ck4-cont}) as $h \to 0$, when $f$ has finite regularity.
    \begin{theorem}\label{Thm_Cheby_Coeff}
		Let $k \geq 1$ be an integer, and let $c_{k}^{(\ell)}$ ($\ell=1,2,3,4$) denote the $k$-th Chebyshev coefficient of $f$ with respect to the $i$-th kind Chebyshev polynomials as defined in \eqref{eq:ck1-cont}--\eqref{eq:ck4-cont}. Then
        \begin{equation}\label{eq:Cheby_bound}
            |c_{k}^{(\ell)}| \le C\dfrac{h^{\min\{k,s\}}}{k^{m+2}} \quad \text{where} \quad s = \begin{dcases}
                m+1, \text{ if } f\in X^m[a,b], \\
                m+2, \text{ if } f\in C^{m+2}[a,b],
            \end{dcases}
        \end{equation}
		$h = b - a$ is the length of the interval $[a,b]$ such that $h \le 1$ and $C$ is a constant independent of $k$, $h$ and $\ell$. 
	\end{theorem}
The proof of Theorem~\ref{Thm_Cheby_Coeff} is presented at the end of this section, following two preparatory lemmas, namely Lemma~\ref{lemma_expression_Ck_induction} and Lemma~\ref{lemma_Cheby_Coeff_Taylor_thm}. Recall that the $k$-th Chebyshev coefficient of the first kind, $c_k$, is defined in \eqref{eq:ck1-cont}. By applying the substitution $x=\cos(\theta)$ and performing integration by parts, one can express $c_k$ as
	\begin{equation}\label{Eq_Cheby_Coeff_intPart_one} 
		c_k = \frac{1}{\pi}\int_{-\pi}^{\pi}\widetilde{f}(\cos{\theta})\cos(k\theta)d\theta = \frac{h}{2k}\frac{1}{\pi}\int_{-\pi}^{\pi}\widetilde{f}'(\cos{\theta})\sin(\theta)\sin{(k\theta)}d\theta.  
	\end{equation}
As the aim of this section is to derive a decay estimate of $c_k$ in terms of $h$ and $k$, the above equation \eqref{Eq_Cheby_Coeff_intPart_one} shows that $c_k$ can be expressed as the powers of $k$ and $h$. To obtain the optimal powers of $h$ and $k$, one of the main ideas is to perform the integration parts iteratively. To simplify calculations, we now define the following notation:
	\begin{equation}\label{Ck_int_by_parts_term}
		I_{M}[f,(i,j,k)] = \int_{-\pi}^{\pi} \widetilde{f}^{(i)}(\cos{\theta})\sin^{j}(\theta)\cos^{i-j}(\theta)\zeta_{M}(k\theta)d\theta,  
	\end{equation}
	where $i,M$ are non-negative integers, and $j$ is an integer such that $j\le i$. If $j<0$, then we define $I_{M}[f,(i,j,k)] = 0$. Note that $M$ corresponds to the number of times integration by parts is performed and  $\zeta_M(\theta) = \cos{\theta}$ if $M$ is even, else $\zeta_M(\theta) = \sin{\theta}$. The notation $\widetilde{f}^{(i)}$ represents the $i$-th derivative of $\widetilde{f}$, with $i$ determined by the regularity of $\widetilde{f}$. For $f \in X^m$ (see \eqref{eq:Xm-def}), the following lemma provides a recursive relation for $I_{M}[f,(i,j,k)]$ (as defined in \eqref{Ck_int_by_parts_term}), which serves as an important tool in deriving bounds for $c_k$.
	\begin{lemma}\label{lemma_expression_Ck_induction}
		If $f\in X^{m}[a,b]$ then for integers $i$, $j$, with $0\leq j \le i\leq m$ and $k\in \N$, the integral $I_{M}[f,(i,j,k)]$ defined in \eqref{Ck_int_by_parts_term} satisfies the following recursive relation:
        \begin{align}\label{eq_rec_reln}
I_{M}[f,(i,j,k)] 
  = \frac{(-1)^M}{k}\Big[&
        - j \, I_{M+1}[f,(i,j-1,k)]  \notag \\
  & + (i-j)\, I_{M+1}[f,(i,j+1,k)] \\
  & + \tfrac{h}{2}\, I_{M+1}[f,(i+1,j+1,k)]
     \;\Big]. \notag
\end{align}
	Moreover, if $f\in C^{m+2}[a,b]$ then equation \eqref{eq_rec_reln} holds for $0\leq j\leq i\leq m+1$.
	\end{lemma}
	\begin{proof}
		Using integration by parts, and the relation 
		\begin{equation*}
			\int \zeta_M(k\theta)d\theta = \frac{(-1)^M}{k}\zeta_{M+1}(k\theta),
		\end{equation*}
		for $0\leq i \leq m-1$, $I_{M}[f,(i,j,k)]$ can be written as
		\begin{align} 
			I_{M}[f,(i,j,k)] &= \int_{-\pi}^{\pi} \widetilde{f}^{(i)}(\cos{\theta})\sin^{j}(\theta)\cos^{i-j}(\theta)\zeta_{M}(k\theta)d\theta \nonumber \\
			&= \left(\widetilde{f}^{(i)}(\cos{\theta})\sin^{j}(\theta)\cos^{i-j}(\theta)\frac{(-1)^M}{k}\zeta_{M+1}(k\theta)\right)_{-\pi}^{\pi} \nonumber\\
			&\quad + \frac{(-1)^{M+1}}{k}\int_{-\pi}^{\pi} \left(\widetilde{f}^{(i)}(\cos{\theta})\sin^{j}(\theta)\cos^{i-j}(\theta) \right)' \zeta_{M+1}(k\theta)d\theta \label{int_by_parts_step}
        \end{align}
        which can be further simplified as
        \begin{align}
			I_{M}[f,(i,j,k)] &= \frac{(-1)^{M}}{k} \frac{h}{2}\int_{-\pi}^{\pi} \widetilde{f}^{(i+1)}(\cos{\theta})\sin^{j+1}(\theta)\cos^{i-j}(\theta) \zeta_{M+1}(k\theta)d\theta \nonumber\\
			&\quad +\frac{(-1)^{M+1} j}{k}\int_{-\pi}^{\pi}\widetilde{f}^{(i)}(\cos{\theta})\sin^{j-1}(\theta)\cos^{i-(j-1)}(\theta)  \zeta_{M+1}(k\theta)d\theta \nonumber\\
			&\quad +\frac{(-1)^{M}}{k}(i-j)\int_{-\pi}^{\pi} \widetilde{f}^{(i)}(\cos{\theta})\sin^{j+1}(\theta)\cos^{i-(j+1)}(\theta) \zeta_{M+1}(k\theta)d\theta .
		\end{align}
		Hence, for $f\in X^{m}[a,b]$ we obtain the recurrence relation~\eqref{eq_rec_reln} for $1 \le i \leq m-1$.
		Since $f^{(m)}$ is not differentiable in $[a,b]$ and $\widetilde{f}(\cos{(\cdot)})\in C^{m+2}_{\mathrm{pw}}[-\pi,\pi]$, there exist discontinuities of $(m+1)$-th derivative of $\widetilde{f}(\cos{\theta})$, namely 
        \begin{equation}\label{eq:discont_pts}
         \{s_{\ell}\}_{\ell=0}^{n_d}\quad \text{for some}\quad n_{d}\in \N \quad   \text{such that}\quad [-\pi,\pi] = \bigcup_{\ell=1}^{n_d} [s_{\ell-1},s_\ell].
        \end{equation}
         Therefore, $\widetilde{f}(\cos{(\cdot)})\in C^{m+2}(s_{\ell-1},s_{\ell})$ for each $\ell$. Using this decomposition, an application of integration by parts one more time gives
		\begin{align}
			I_{M}[f,(m,j,k)] &= \int_{-\pi}^{\pi} \widetilde{f}^{(m)}(\cos{\theta})\sin^{j}(\theta)\cos^{i-j}(\theta)\zeta_{m}(k\theta)d\theta \nonumber\\
			&= \sum_{\ell=1}^{n_d}\int_{s_{\ell-1}}^{s_{\ell}} \widetilde{f}^{(m)}(\cos{\theta})\sin^{j}(\theta)\cos^{i-j}(\theta)\zeta_{m}(k\theta)d\theta \nonumber\\
			&= \frac{(-1)^m}{k}\sum_{\ell=1}^{n_d}\left[ \left(\widetilde{f}^{(m)}(\cos{\theta})\sin^{j}(\theta)\cos^{i-j}(\theta) \zeta_{m+1}(k\theta)\right)_{s_{\ell-1}}^{s_\ell} \right. \nonumber\\
			&\left.\hspace{1cm} -\int_{s_{\ell-1}}^{s_{\ell}} \left(\widetilde{f}^{(m)}(\cos{\theta})\sin^{j}(\theta)\cos^{i-j}(\theta)\right)'\zeta_{m+1}(k\theta)d\theta \right]. \label{intbyparts_discont}
		\end{align}
		Using the continuity of $f^{(m)}$, the first term in \eqref{intbyparts_discont} becomes zero. Indeed, 
		\[ \sum_{\ell=1}^{n_d}\left(\widetilde{f}^{(m)}(\cos{\theta})\sin^{j}(\theta)\cos^{i-j}(\theta)\zeta_{m+1}(k\theta)\right)_{s_{\ell-1}}^{s_\ell} = \left(\widetilde{f}^{(m)}(\cos{\theta})\sin^{j}(\theta)\cos^{i-j}(\theta)\zeta_{m+1}(k\theta)\right)_{-\pi}^{\pi} = 0.\]
		Since the discontinuities of $\widetilde{f}^{(m+1)}$ are finitely many, the second term in \eqref{intbyparts_discont} can be simplified as 
		\begin{align*}
			\sum_{\ell=1}^{n_d} \int_{s_{\ell-1}}^{s_{\ell}} \left(\widetilde{f}^{(m)}(\cos{\theta})\sin^{j}(\theta)\cos^{i-j}(\theta)\right)'\zeta_{m+1}(k\theta)d\theta =  \int_{-\pi}^{\pi} \left(\widetilde{f}^{(m)}(\cos{\theta})\sin^{j}(\theta)\cos^{i-j}(\theta)\right)'\zeta_{m+1}(k\theta)d\theta.
		\end{align*}
		This implies that we obtained the same recursive relation in this case. Moreover, if $f\in C^{m+2}[a,b]$, it is easy to see that equation \eqref{eq_rec_reln} holds even for $0\leq i\leq m+1$.
	\end{proof}
	\noindent
	Now, using the Lemma~\ref{lemma_expression_Ck_induction} and  \eqref{Ck_int_by_parts_term}, $c_{k}$ in \eqref{Eq_Cheby_Coeff_intPart_one} can be re-expressed as $c_{k} = \frac{1}{2\pi} \frac{h}{k} I_{1}[f,(1,1,k)]$. Application of \Cref{lemma_expression_Ck_induction} yields
     \begin{equation}
         c_{k} = \frac{1}{2\pi} \frac{h}{k} I_{1}[f,(1,1,k)] = \frac{1}{2\pi}\frac{h}{k^2}I_{2}[f,(1,0,k)]- \frac{1}{4\pi}\frac{h^2}{k^2}I_{2}[f,(2,2,k)]
     \end{equation}
By performing integration by parts iteratively $M$ times, we obtain the relation stated in Lemma~\ref{lemma_Cheby_Coeff_form}:
\begin{equation}\label{eq_Cheby_Coeff_form1}
			c_k = \sum_{i=1}^{M} \sum_{j=0}^{i} \frac{h^{i}}{k^M}\alpha_{ij}^{(M)}  I_{M}[f,(i,j,k)], \quad \text{for} \quad 1 \le M \le s \quad \text{where} \quad s = \begin{dcases}
                m+1, \text{ if } f\in X^m[a,b], \\
                m+2, \text{ if } f\in C^{m+2}[a,b].
            \end{dcases}
		\end{equation}
While estimating the decay rate of $c_{k}$ in terms of $k$ and $h$, we now have obtained the optimal decay in $k$ as in \eqref{eq_Cheby_Coeff_form1}.
We now proceed to analyze the corresponding optimal dependence on $h$. To do so, we re-write the integral \eqref{Ck_int_by_parts_term} as a function of $\mu$,
\begin{equation}\label{Eq_expression_gijkl}
		I_{M}[f,(i,j,k)](\mu) = \int_{-\pi}^{\pi} \widetilde{f}^{(i)}\left(\frac{\mu}{h}\cos{\theta}\right)\sin^{j}(\theta)\cos^{i-j}(\theta)\zeta_{M}(k\theta)d\theta,
	\end{equation}
 and investigate its behavior as $\mu \to 0$. Here $\mu\in [-h, h]$ and observe that $I_{M}[f,(i,j,k)](h) = I_{M}[f,(i,j,k)]$. In the following lemma we estimate a bound for \eqref{Eq_expression_gijkl}.  
	\begin{lemma}\label{lemma_Cheby_Coeff_Taylor_thm}
		For integers $0 \leq j\leq i \leq m$ with $i<k$ the following holds:  
        \begin{equation}
            I_M[f,(i,j,k)](\mu) = O\left(\mu^{\min\{k, s\}-i}\right) \quad \text{where} \quad s = \begin{dcases}
                m+1, \text{ if } f\in X^m[a,b], \\
                m+2, \text{ if } f\in C^{m+2}[a,b],
            \end{dcases}
        \end{equation}
		and $\mu \in [-h, h] \subseteq [-1,1]$.
	\end{lemma}
	\begin{proof}
    We first investigate the regularity of the integral $I_M[f,(i,j,k)](\mu)$ defined in \eqref{Eq_expression_gijkl} at $\mu=0$ and assume $f\in X^{m}[a,b]$. Using Leibniz rule we deduce that $I_{M}[f,(i,j,k)](\mu)$ is $(m-i)$ times differentiable. In particular the derivatives of $I_{M}[f,(i,j,k)](\mu)$ are given by 
		\begin{equation}\label{eq_derivative_of_integral}
			\frac{d^{n}}{d\mu^{n}}I_{M}[f,(i,j,k)](\mu) = \frac{1}{2^n}\int_{-\pi}^{\pi}\widetilde{f}^{(i+n)}\left(\frac{\mu}{h}\cos{\theta}\right)\sin^{j}(\theta)\cos^{i-j+n}(\theta)\zeta_{M}(k\theta)d\theta = \frac{1}{2^n}I_{M}[f,(i+n,j,k)](\mu),
		\end{equation}
		for $0 \le n \le m-i$. Since $f \in C_{pw}^{m+2}[a,b]$,  $(m+1-i)^{th}$ derivative of $I_{M}[f,(i,j,k)](\mu)$ exists. Indeed, the points of discontinuities occur at $\{s_\ell\}_{\ell=0}^{n_d}$ as given in~\eqref{eq:discont_pts}. Therefore, the $(m-i)$-th derivative can be written as
		\begin{equation}
			\frac{d^{m-i}}{d\mu^{m-i}}I_{M}[f,(i,j,k)](h) = \frac{1}{2^{m-i}}\sum_{q=1}^{N_d}\int_{s_{q-1}(\mu)}^{s_q(\mu)} \widetilde{f}^{m}\left(\frac{\mu}{h}\cos{\theta}\right)\sin^{j}(\theta)\cos^{m-j}(\theta)\zeta_{M}(k\theta)d\theta,
		\end{equation}
		and we can apply Leibniz rule again to show that $I_{M}[f,(i,j,k)](\mu)$ is $(m+1-i)$ times differentiable. Expanding $I_{M}[f,(i,j,k)](\mu)$ in a finite Taylor series about $\mu=0$, we deduce
		\begin{equation}\label{eq_Taylor_expansion_Xm+2}
			I_{M}[f,(i,j,k)](\mu) = \sum_{n = 0}^{m-i} I_{M}[f,(i+n,j,k)](0) \frac{\mu^{n}}{2^n n!} + O\left(\mu^{m+1-i}\right).
		\end{equation}
		We now estimate the right hand side of equation \eqref{eq_Taylor_expansion_Xm+2}, by employing Lemma~\ref{Ap1} in Appendix, for $0\le n \le m-i$ we obtain 
		\begin{equation}\label{eq_int_powsincos}
			I_{M}[f,(i+n,j,k)](0) = \widetilde{f}^{(i+n)}(0)\int_{-\pi}^{\pi} \sin^j{(\theta)}\cos^{i+n-j}{(\theta)}\zeta_M(k\theta) d\theta = 0,  
		\end{equation}
		if $0\le j \le i+n < k $ or $n<k-i$. Therefore, if $k \le m$, the first non-zero term of the summation in equation \eqref{eq_Taylor_expansion_Xm+2} occur at $n=k-i$ and we obtain 
        \begin{equation*}
        I_M[f,(i,j,k)](\mu) = O(\mu^{k-i}).    
        \end{equation*}
        If $k > m$, then every term in the summation of \eqref{eq_Taylor_expansion_Xm+2} becomes zero because of \eqref{eq_int_powsincos}, which implies 
        \begin{equation*}
            I_M[f,(i,j,k)](\mu) = O(\mu^{m+1-i}).    
        \end{equation*}
        Combining these two cases we get the desired result when $f\in X^{m}[a,b]$. For $f\in C^{m+2}[a,b]$, using Leibniz rule we deduce that $I_{M}[f,(i,j,k)](\mu)$ is $(m+2-i)$ times differentiable. In particular the derivatives of $I_{M}[f,(i,j,k)](\mu)$ are given as~\eqref{eq_derivative_of_integral} for $0 \le n \le m+2-i$. By employing a finite Taylor expansion of $I_{M}[f,(i,j,k)](\mu)$ at $\mu=0$, we have
		\begin{equation*}
			I_{M}[f,(i,j,k)](\mu) = \sum_{n = 0}^{m+1-i} I_{M}[f,(i+n,j,k)](0) \frac{\mu^{n}}{2^n n!} + O\left(\mu^{m+2-i}\right). 	
		\end{equation*}
		Now, following the similar arguments which we carried out for $f\in X^{m}[a,b]$, we obtain the desired bound
        \begin{equation*}
         I_{M}[f,(i,j,k)](\mu) = O(\mu^{k-i}) + O(\mu^{m+2-i})   
        \end{equation*}
     if $f\in C^{m+2}[a,b]$.
	\end{proof}

\begin{center}{\bf Proof of Theorem~\ref{Thm_Cheby_Coeff}}
\end{center}
It suffices to prove Theorem~\ref{Thm_Cheby_Coeff} for $\ell=1$, since the Chebyshev coefficients $c_k^{(2)},~c_k^{(3)},~c_k^{(4)}$ corresponding to second, third and fourth kind Chebyshev polynomials respectively, can be expressed in terms of $c_k^{(1)}$ using equations \eqref{eq:ck1-cont}-\eqref{eq:ck4-cont}. In particular, we have
	\begin{equation}\label{eq:ck2-using-ck1}
		c_k^{(2)} = \frac{1}{\pi} \int_{0}^{\pi} f(\cos{\theta})\cos{(k\theta)}d\theta-\frac{1}{\pi} \int_{0}^{\pi} f(\cos{\theta})\cos{((k+2)\theta)}d\theta = \frac{c_k^{(1)}}{\gamma_k} - \frac{c_{k+2}^{(1)}}{2},
	\end{equation}
	\begin{equation}\label{eq:ck3-using-ck1}
		c_k^{(3)} = \frac{1}{\pi} \int_{0}^{\pi} f(\cos{\theta})\cos{(k\theta)}d\theta + \frac{1}{\pi} \int_{0}^{\pi} f(\cos{\theta})\cos{((k+1)\theta)}d\theta  = \frac{c_k^{(1)}}{\gamma_k} + \frac{c_{k+1}^{(1)}}{2},
	\end{equation}
	\begin{equation}\label{eq:ck4-using-ck1}
		c_k^{(4)} = \frac{1}{\pi} \int_{0}^{\pi} f(\cos{\theta})\cos{(k\theta)}d\theta-\frac{1}{\pi} \int_{0}^{\pi} f(\cos{\theta})\cos{((k+1)\theta)}d\theta  = \frac{c_k^{(1)}}{\gamma_k} - \frac{c_{k+1}^{(1)}}{2}.
	\end{equation}
For simplicity of notation we write $c_k$ in place of $c_k^{(1)}$ throughout the rest of the proof. If $f\in X^{m}[a,b]$, using~\eqref{eq_Cheby_Coeff_form1} for $M=m+1$, an expression for $c_{k}$ can be obtained as
	\begin{align}
		c_k &= \sum_{i=1}^{m+1} \sum_{j=0}^{i} \frac{h^{i}}{k^{m+1}}\alpha_{ij}^{(m+1)}  I_{m+1}[f,(i,j,k)]\\
		&= \sum_{i=1}^{m} \sum_{j=0}^{i} \alpha_{ij}^{(m+1)} \frac{h^{i}}{k^{m+1}} I_{m+1}[f,(i,j,k)]+\sum_{j=0}^{m+1} \alpha_{(m+1)j}^{(m+1)} \frac{h^{m+1}}{k^{m+1}} I_{m+1}[f,(m+1,j,k)].\label{Eq_Cheby_Coeff_IBP_m+1}
	\end{align}
	Additionally, in \Cref{lemma_Cheby_Coeff_form} we have shown that $I_{m+1}[f,(m+1,j,k)] =O\left(1/k\right)$. Using this bound, the second term in~\eqref{Eq_Cheby_Coeff_IBP_m+1} can be estimated as 
	\begin{equation*}
		\left|\sum_{j=0}^{m+1} \alpha_{(m+1)j}^{(m+1)} \frac{h^{m+1}}{k^{m+1}} I_{m+1}[f,(m+1,j,k)]\right| \le C\frac{h^{m+1}}{k^{m+2}}.
	\end{equation*}
	Using the recurrence relation for $I_{m+1}[f,(i,j,k)]$ described in \Cref{lemma_expression_Ck_induction}, the first sum in equation~\eqref{Eq_Cheby_Coeff_IBP_m+1} can be expressed as follows:
	\begin{align*} 
		c_k &= \sum_{i=1}^{m} \sum_{j=0}^{i} \alpha_{ij}^{(m+1)} \frac{h^{i}}{k^{m+1}} \left[ (-1)^{m}\frac{j}{k}I_{m+2}[f,(i,j-1,k)]-(-1)^m \frac{(i-j)}{k}I_{m+2}[f,(i,j+1,k)] \right. \\ &\quad \left.-\frac{(-1)^m}{2}\frac{h}{k}I_{m+2}[f,(i+1,j+1,k)]\right]\\
		&= \sum_{i=1}^{m} \frac{h^{i}}{k^{m+2}}\sum_{j=0}^{i} \alpha_{ij}^{(m+1)} \left[ (-1)^{m}j I_{m+2}[f,(i,j-1,k)]-(-1)^m (i-j)  I_{m+2}[f,(i,j+1,k)]\right. \\
		&\quad  \left. - \frac{(-1)^m}{2}h I_{m+2}[f,(i+1,j+1,k)]\right].
	\end{align*}
	Thus, using \Cref{lemma_Cheby_Coeff_Taylor_thm} and observing the fact $I_{M}[f,(i,j,k)](h) = I_{M}[f,(i,j,k)]$. we obtain 
	\begin{align*}
			\left|I_{m+2}[f,(i,j-1,k)]\right| &\le  C \, h^{\min\{k,m+1\}-i}, \\
			\left|I_{m+2}[f,(i,j+1,k)]\right| &\le C\, h^{\min\{k,m+1\}-i}, \\
			\left|I_{m+2}[f,(i+1,j+1,k)]\right| &\le C \, h^{\min\{k,m+1\}-(i+1)}.
	\end{align*}
	Since the constants in the above summations are dependent only on $i$, $j$ and $m$, and independent of $h$ and $k$ we obtain the required result. Similarly, if $f\in C^{m+2}[a,b]$, using~\eqref{Eq_Cheby_Coeff_IBP_m+1} for $M=m+2$, $k$-th coefficient $c_{k}$ can be expressed as
	\begin{equation}\label{eq_Cheby_coeff_uCm+2}
		c_k = \sum_{i=1}^{m+2} \sum_{j=0}^{i} \frac{h^{i}}{k^{m+2}}\alpha_{ij}^{(m+2)}  I_{m+2}[f,(i,j,k)].
	\end{equation}
	Using \Cref{lemma_Cheby_Coeff_Taylor_thm} and the fact $I_{m+2}[f,(i,j,k)] = I_{m+2}[f,(i,j,k)](h)$, we obtain
	\begin{equation}
		\left|I_{m+2}[f,(i,j,k)]\right| \le C \, h^{\min\{k,m+2\}-i},
	\end{equation}
	and since $\alpha_{ij}^{m+2}$ in equation \eqref{eq_Cheby_coeff_uCm+2} depend only on $i$, $j$ and $m$, the required bound for $c_{k}$ is attained when $f\in C^{m+2}[a,b]$. Therefore, we have 
    \begin{equation}\label{eq:Cheby_bound1}
            |c_{k}| \le C\dfrac{h^{\min\{k,s\}}}{k^{m+2}} \quad \text{where} \quad s = \begin{dcases}
                m+1, \text{ if } f\in X^m[a,b], \\
                m+2, \text{ if } f\in C^{m+2}[a,b].
            \end{dcases}
    \end{equation}

    \section{Applications}\label{sec:appln}
	In this section, we present the application of Theorem~\ref{Thm_Cheby_Coeff} which characterizes the decay of Chebyshev coefficients of certain class of functions. Specifically, we present three applications: (i) error estimates for approximating continuous Chebyshev coefficients~\eqref{eq:ck1-cont}-\eqref{eq:ck4-cont} by their discrete counterparts presented in Table~\ref{T:quadratures}, (ii) decay rates of the discrete Chebyshev coefficients, and (iii) error bounds for the interpolatory quadratures introduced in Definition~\ref{sec:prelim}.
    
	As stated at the end of Section~\ref{sec:prelim}, denote the symbol $C$ for a positive constant taking on different values on different occurrences. We now introduce the notations for the discrete Chebyshev coefficients and the interpolatory quadratures based on Chebyshev polynomials introduced in Definition~\ref{def:quad} and Table~\ref{T:quadratures}. For $f\in C[-1,1]$, let
	\[\widetilde{c}_{k}^{f1},~ \widetilde{c}_{k}^{cc},~\widetilde{c}_{k}^{f2},~
	\widetilde{c}_{k}^{f3},~\widetilde{c}_{k}^{f4},\qquad \text{and} \qquad P_n^{f1}[f],~ P_n^{cc}[f],~P_n^{f2}[f],~
	P_n^{f3}[f],~P_n^{f4}[f]\] 
	denote the discrete Chebyshev 
	coefficients and the interpolating polynomials corresponding to the nodes of the Fej\'er first (F-I), 
	Clenshaw--Curtis (CC), Fej\'er second (F-II), Fej\'er third (F-III),and Fej\'er fourth (F-IV) quadratures, respectively. The discrete coefficients are defined in Table~\ref{T:quadratures} and the interpolating polynomials are given by
    \begin{equation}\label{eq:inter_polygen}
\begin{aligned}
P_n^{f1}[f](t) = \sum_{k=0}^{n-1}\widetilde{c}_{k}^{f1}T_k(t), 
\quad P_n^{cc}[f](t) = &\sum_{k=0}^{n-1}\widetilde{c}_{k}^{cc}T_k(t), 
\quad P_n^{f2}[f](t) = \sum_{k=0}^{n-1}\widetilde{c}_{k}^{f2}U_k(t), \\[6pt]
P_n^{f3}[f](t) = \sum_{k=0}^{n-1}\widetilde{c}_{k}^{f3}V_k(t), \quad &\text{and}
\quad P_n^{f4}[f](t) = \sum_{k=0}^{n-1}\widetilde{c}_{k}^{f4}W_k(t),
\end{aligned}
\end{equation}
	respectively. Similarly, let 
	\[Q_{n}^{f1}[f],~Q_{n}^{cc}[f],~Q_{n}^{f2}[f],~Q_{n}^{f3}[f],~Q_{n}^{f4}[f]\] 
	denote the $n$-point F-I, CC, F-II, 
	F-III, and F-IV quadrature approximations (as defined in~\eqref{eq:quad}) 
	of the integral of a function $f$ over an interval $[a,b]$, 
	with the corresponding quadrature weights and nodes listed in 
	Table~\ref{T:quadratures}. 
	For convenience, we also use the compact notation 
    \begin{equation}\label{eq:Qnq_not}
        Q_{n}^{q}[f], \quad  \text{where} \quad q \in \{f1, cc, f2, f3, f4\}.
    \end{equation}
    In view of these notations, the following two theorems present error estimates for the approximation of continuous Chebyshev coefficients by their discrete counterparts.
	

%

\begin{theorem}\label{thm:error_discrete_and_continuous_ChebyCoeff_Tn}
Let $k$ be an integer with $0 \leq k \leq n-1$, and let $c_k^{(\ell)}$ ($\ell=1,2,3,4$) denote the $k$-th Chebyshev coefficient of $f$, as defined in~\eqref{eq:ck1-cont}--\eqref{eq:ck4-cont}. Then
\begin{equation}
    \bigl|c_k^{(\ell)} - \widetilde{c}_k^{\,fi}\bigr| \;\le\; 
    C \, \frac{h^{\min\{n+1,\,s\}}}{n^{m+2}}, \quad \text{where} \quad s = \begin{dcases}
                m+1, \text{ if } f\in X^m[a,b], \\
                m+2, \text{ if } f\in C^{m+2}[a,b],
            \end{dcases}
\end{equation}
$h = b-a \leq 1$ and $n$ is the number of terms in the discrete Chebyshev expansion $P_n^{fi}[f]$ given in~\eqref{eq:inter_polygen}.
\end{theorem}

	\begin{proof}
			We prove the theorem for $\ell=1$ and for rest of the cases the proof follows analogously. We first consider the case $f\in X^m[a,b]$. Using the notation $\widetilde{f}$ as in \eqref{eq:util_def}, the Chebyshev series expansion of $f$ is given by
		\begin{equation}\label{eq:util_with_Rn_Tn}
			\widetilde{f}(t) =  \sum_{k=0}^{\infty} c_k^{(1)} T_k(t) =\sum_{k=0}^{n-1} c_k^{(1)} T_k(t) + R_n(t), \quad \text{where} \quad R_n(t) = \sum_{k=n}^{\infty} c_k^{(1)} T_k(t).
		\end{equation}
		From~\eqref{eq:util_with_Rn_Tn} it is clear that $(\widetilde{f}-R_n)(t)$ is a polynomial with degree at most $n-1$ implying that discrete and continuous Chebyshev coefficients of $\widetilde{f}-R_n$ are equal. We obtain
		\begin{equation}
			(\widetilde{f}-R_n)(t) = \sum_{k=0}^{n-1} c_k^{(1)} T_k(t), \quad \text{where} \quad c_k^{(1)}= \frac{\gamma_k}{n} \sum_{j=0}^{n-1} (\widetilde{f}-R_n)(t_j)T_k(t_j),
		\end{equation}
		and $t_j = \cos{\left(\frac{2j+1}{2n}\pi\right)}$, $0 \le j \le n-1$.
		The error in the continuous and the discrete Chebyshev coefficients can be expressed as
		\begin{equation}\label{eq:ck-cktilde}
			c_{k}^{(1)}-\widetilde{c}_{k}^{f1} =\frac{\gamma_{k}}{n} \sum\limits_{j=0}^{n-1}\left(\widetilde{f}-R_{n}\right)\left(t_{j}\right) T_{k}\left(t_{j}\right)-\frac{\gamma_{k}}{n} \sum\limits_{j=0}^{n-1} \widetilde{f}\left(t_{j}\right) T_{k}\left(t_{j}\right) =-\frac{\gamma_{k}}{n} \sum_{j=0}^{n-1} R_{n}\left(t_{j}\right) T_{k}\left(t_{j}\right).    
		\end{equation}
		Since the Chebyshev series is absolutely and uniformly convergent, we interchange the infinite summation in~\eqref{eq:ck-cktilde} to obtain
		\begin{equation}\label{eq:apply_dis_otho}
			\widetilde{c}_{k}^{f1}-c_{k}^{(1)}=\frac{\gamma_{k}}{n} \sum_{j=0}^{n-1}\left[\sum_{i=n}^{\infty} c_{i}^{(1)} T_{i}\left(t_{j}\right)\right] T_{k}\left(t_{j}\right) = \frac{\gamma_{k}}{n} \sum_{i=n}^{\infty} c_{i}^{(1)}\left[\sum_{j=0}^{n-1} T_{i}\left(t_{j}\right) T_{k}\left(t_{j}\right)\right].
		\end{equation}
		Using the discrete orthogonality relation~\eqref{eq:dis_otho_FFQ} from Lemma~\ref{lem:dis_otho_gen}, the finite sum in the rightmost equality of~\eqref{eq:apply_dis_otho} simplifies, and the error term can be expressed as
		\begin{equation}   
			\widetilde{c}_{k}^{f1}-c_{k}^{(1)} = \frac{\gamma_{k}}{2}\left[\sum_{l=1}^{\infty} (-1)^{l} c^{(1)}_{2 n l-k}+\sum_{l=1}^{\infty}(-1)^{l} c^{(1)}_{2nl+k}\right].
		\end{equation} 
		For every $0\leq k < n$, $|\gamma_{k}|\leq 2$ and using Theorem~\ref{Thm_Cheby_Coeff}, we deduce
		\begin{equation}\label{eq:crude_bound_Tn}
			\left|c_{k}^{(1)}-\widetilde{c}_{k}^{f1}\right| \leq C \sum_{l=1}^{\infty}\left[\frac{h^{\min\{2nl-k,m+1\}}}{(2 nl-k)^{m+2}}+\frac{h^{\min\{2nl+k,m+1\}}}{(2 n l+k)^{m+2}}\right],
		\end{equation}
		where $C$ is a positive constant independent of $h$ and $n$. For $l \ge 1$ and $0 \le k\le n-1$ , $2nl-k \ge n+1$ and $2nl+k \ge 2n$, therefore, the required power of $h$, which is $h^{\min\{n+1,m+1\}}$ can be factored out of the sum given on the right hand side of the inequality~\eqref{eq:crude_bound_Tn}, and the rest of the higher powers of $h$ can be bounded since $h\le 1$. Thus, we obtain
		\begin{equation}
			\left|c_{k}^{(1)}-\widetilde{c}_{k}^{f1}\right| \leq C\frac{h^{\min\{n+1,m+1\}}}{n^{m+2}} \sum_{l=1}^{\infty}\left[\frac{1}{\left(2l-\frac{k}{n}\right)^{m+2}}+\frac{1}{\left(2l+\frac{k}{n}\right)^{m+2}}\right].
		\end{equation}
		For $0 \leq k \leq n-1$, the terms $1/\left(2l -k/n\right)^{m+2}$ and $1/\left(2l +k/n\right)^{m+2}$ are bounded by $1/(2l-1)^{m+2}$ and $1/(2l)^{m+2}$ respectively. Thus, the infinite series is bounded and the result follows. The proof when $f\in C^{m+2}$ can be carried out in a similar manner.
	\end{proof}

	\begin{theorem}\label{thm:error_discrete_and_continuous_ChebyCoeff_CCQ}
		Let $k$ be an integer such that $0 \le k \le n-1$ and $c_k^{(1)}$ denote the $k$-th Chebyshev coefficient of $f$ as defined in~\eqref{eq:ck1-cont}, then the following holds:
        \begin{equation}
            |c_k^{(1)} - \widetilde{c}_k^{cc}| \le \frac{C}{n^{m+2}}\begin{dcases} 
					h^{\min\{n+1,s\}}&, \text{ if } k\ne n-2, \\
					h^{\min\{n,s\}}&, \text{ if } k= n-2,
				\end{dcases} \quad \text{where} \quad s = \begin{dcases}
                m+1, \text{ if } f\in X^m[a,b], \\
                m+2, \text{ if } f\in C^{m+2}[a,b],
            \end{dcases}
        \end{equation}
		$h = b-a \le 1$, and $n\ge 2$ is the number of terms in the discrete Chebyshev expansion $P_n^{cc}[f]$ given in~\eqref{eq:inter_polygen}.
	\end{theorem}

	 In the following corollary, we present the decay rates for the discrete Chebyshev coefficients. 
	\begin{corollary}\label{cor:disCheby_Coeff_Tn}
		Let $n$ be a positive integer and $k$ be an integer such that $0 \le k \le n-1$. Then
        \begin{equation}
            \left|\widetilde{c}_{k}\right| \le C \frac{h^{\min\{k,s\}}}{k^{m+2}}, \quad \text{where} \quad s = \begin{dcases}
                m+1, \text{ if } f\in X^m[a,b], \\
                m+2, \text{ if } f\in C^{m+2}[a,b],
            \end{dcases}
        \end{equation}
		
		$\widetilde{c}_{k}$ is any of the $\widetilde{c}_{k}^{f1},\widetilde{c}_{k}^{cc},\widetilde{c}_{k}^{f2},\widetilde{c}_{k}^{f3},\widetilde{c}_{k}^{f4}$ and $h = b-a \le 1$. 
	\end{corollary}
	\begin{proof}
	  Using the triangle inequality, 
\begin{equation}\label{eq:triangle}
    |\widetilde{c}_k| \leq |c_k - \widetilde{c}_k| + |c_k|,
\end{equation}
and combining the decay of $c_k$ established in Theorem~\ref{Thm_Cheby_Coeff} with the bounds for $|c_k - \widetilde{c}_k|$ from Theorems~\ref{thm:error_discrete_and_continuous_ChebyCoeff_Tn}--\ref{thm:error_discrete_and_continuous_ChebyCoeff_CCQ}, we obtain the desired decay rates for the discrete coefficients $\widetilde{c}_k$.
\end{proof}

As an application of Theorems~\ref{thm:error_discrete_and_continuous_ChebyCoeff_Tn}--\ref{thm:error_discrete_and_continuous_ChebyCoeff_CCQ}, and making use of the interpolatory nature of the quadratures F-I, CC, F-II, F-III, and F-IV (see Table~\ref{T:quadratures}), we establish the following convergence result for the approximation of integral of a function $f$ on the interval $[a,b]$.

	\begin{theorem}\label{thm:quad_convergence_gen}
		Let $n$ be a positive integer, and let $Q_n^q[f]$ denote the $n$-point quadrature rule of type $q \in \{f1,cc,f2,f3,f4\}$ as defined in~\eqref{eq:Qnq_not}. Let $h = b-a \leq 1$ denote the length of the interval $[a,b]$. Then
		\begin{equation}
			\left|\int_{a}^{b} f(y)dy - Q^q_{n}[f] \right| \le \frac{C}{n^{m+2}}\begin{dcases} h^{\min\{n+1+n_0,s+1\}},& \text{ if } q \in\{ f1,cc\},\\[6pt]
			\log{(n)} h^{\min\{n+1+n_0,s+1\}},&\text{ if } q = f2,\\[6pt]
					\log{(n)} h^{\min\{n+1,s+1\}},&\text{ if } q \in\{ f3,f4\}
			\end{dcases}
        \end{equation}
            where
        \begin{equation}
             s = \begin{dcases}
                m+1, \text{ if } f\in X^m[a,b], \\[2pt]
                m+2, \text{ if } f\in C^{m+2}[a,b]
            \end{dcases}, \quad \text{and} \quad
            n_0 = \begin{dcases}
                1, \text{ if } $n$ \text{ is odd},\\[2pt]
                0, \text{ otherwise}.
            \end{dcases}
		\end{equation}
	\end{theorem}
	\begin{proof}
		We prove the theorem for $Q_{n}^{f1}[f]$ and for rest of the quadratures the proof follows analogously. The error in the quadrature rule can be estimated as
		\begin{align*}
			\left| \int_a^b f(y)dy - Q_{n}^{f1}[f]\right| &=\left| \frac{h}{2}\int_{-1}^1 \widetilde{f}(t)dt - \frac{h}{2}\int_{-1}^{1} P_n^{f1}[\widetilde{f}](t) dt\right| \\
			&= \left| \frac{h}{2}\int_{-1}^1 \sum_{k=0}^{\infty}c_k^{(1)} T_k(t) dt - \frac{h}{2}\int_{-1}^{1} \sum_{k=0}^{n-1}\widetilde{c}_k^{f1}T_k(t) dt\right|\\
			&\le \frac{h}{2}\left[\sum_{k=n+n_0}^{\infty}|c_k^{(1)}|\left|\int_{-1}^1 T_k(t)dt\right|+\sum_{k=0}^{n-1}|c_k^{(1)}-\widetilde{c}_k^{f1}|\left|\int_{-1}^1 T_k(t)dt\right|\right]. 
		\end{align*}
		Using the bounds given in Theorem~\ref{Thm_Cheby_Coeff} and Theorem~\ref{thm:error_discrete_and_continuous_ChebyCoeff_Tn}, 
		\begin{equation}\label{eq:quad-proof}
			\left| \int_a^b f(y)dy - Q_{n}^{f1}[f]\right| \le \frac{h}{2}C\left[\sum_{k=n+n_0}^{\infty}\frac{h^{\min\{k,s\}}}{k^{m+2}}\left|\int_{-1}^1 T_k(t)dt\right|+\frac{h^{\min\{n+1,s\}}}{n^{m+2}} \sum_{k=0}^{n-1}\left|\int_{-1}^1 T_k(t)dt\right|\right]
		\end{equation}
		for some constant $C$ independent of $h$ and $n$. Noticing  that 
		\begin{equation}
			\left|\int_{-1}^1 T_k(t)dt\right| \le \frac{2}{k^2}, \ k \ge 1 \text{ and } \sum_{k=0}^{n-1}\left|\int_{-1}^1 T_k(t)dt\right| \le C,
		\end{equation}
		and substituting in \eqref{eq:quad-proof}, 
        we obtain the required result.
	\end{proof} 
 
In what follows, we present the error analysis for the composite quadrature rules corresponding to F-I, CC, F-II, F-III, and F-IV quadratures, as an application of the convergence results established in Theorem~\ref{thm:quad_convergence_gen}. To this end, we first recall the definition of a composite quadrature rule, followed by a discussion of the associated errors.
	\begin{definition}\label{def:Composite_Quadrature}
		Let $[a,b]$ be an interval, and for $P\geq 1$, let $a = a_0 < a_1 < \cdots < a_P = b$ be a partition of $[a,b]$ into $P$ subintervals (patches) $I_p = [a_{p-1},a_p]$, for $p=1,2,\dots,P$. For a given $n$-point quadrature rule on each subinterval $I_p$, i.e  
		\begin{equation}\label{eq:Composite_Quadrature_not}
			\int_{a_{p-1}}^{a_p} f(t)\,dt = \frac{h_p}{2}\int_{-1}^{1} f(\xi_p(t))dt \approx\frac{h_p}{2}\sum_{j=1}^{n} w_{j}\, f\!\big(\xi_p(t_j)\big), \quad\text{where}\quad \xi_p(t) = \frac{h_p}{2}t+\frac{a_p+a_{p-1}}{2},
		\end{equation}
		$h_p = a_p-a_{p-1}$, $\{t_j\}_{j=1}^n$ and $\{w_j\}_{j=1}^n$ are the quadrature nodes and weights, the corresponding $n$-point composite quadrature rule over $[a,b]$ with $P$ patches is defined by
		\begin{equation}
			\int_a^b f(t)\,dt \approx \sum_{p=1}^P \frac{h_p}{2} \sum_{j=1}^{n} w_{j}\, f\!\big(\xi_p(t_{j})\big).
		\end{equation}
	\end{definition}
	Denote $\I_{n}^{f1}[f]$, $\I_{n}^{cc}[f]$, $\I_{n}^{f2}[f]$, $\I_{n}^{f3}[f]$ and $\I_{n}^{f4}[f]$ as the $n$-point composite Fej\'er first, Clenshaw-Curtis, Fej\'er second, Fej\'er third, and the Fej\'er fourth quadrature approximations of the integral of a given function $f$ over an interval $[a, b]$ with quadrature weights and nodes listed in Table~\ref{T:quadratures}, respectively. For notational simplicity, we use a compact notation $\I_{n}^{q}[f]$ for $q$ in $f1$, $cc$, $f2$, $f3$, and $f4$.
	\begin{corollary}\label{cor:compositequad_convergence_gen}
		Let $n$ be a positive integer and $f \in X^{m}[a,b]$, then, for an equi-spaced partition of $[a,b]$
		\begin{equation}
			\left|\int_{a}^{b} f(y)dy - \I^q_{n}[f] \right| \le \frac{C}{n^{m+2}}\begin{dcases} h^{\min\{n+n_0,m+2\}},&  \text{ if } q \in \{f1,cc\}\\[6pt]
				\log{(n)} h^{\min\{n+n_0,m+2\}},&\text{ if } q = f2,\\[6pt]
				\log{(n)} h^{\min\{n,m+2\}},&\text{ if } q \in \{f3,f4\},
			\end{dcases}
		\end{equation}
		where $n_0 = 1$ if $n$ is odd and $n_0=0$ otherwise, $h = (b-a)/P \le 1$ with $P$ denoting the number of patches, and $C$ is some constant independent of $h$, $P$ and $n$. 
	\end{corollary}
	\begin{proof}
		We prove the theorem for the Fej\'er first quadrature and for rest of the quadratures the proof follows analogously.  Let $a = a < a+h < \cdots < a+Ph = b$ be a partition of $[a,b]$ into $P$ sub-intervals where $h=\frac{b-a}{P}$, and $\xi_p(t) = a+ph+\frac{h}{2}\left(t-1\right)$ for $-1\le t\le 1$, $1 \le p \le P$. Then, 
		\begin{equation}
			\left|\int_{a}^{b} f(y)dy - \I^{f1}_{n}[f] \right| \le \sum_{p=1}^P \left|\frac{h}{2}\int_{-1}^{1} f(\xi_p(t))dt - Q_n^{f1}[f(\xi_p)] \right|.
		\end{equation}
        Now, to estimate the total error, observe that $f \in X^m[a,b]$. This implies that $f^{(m+2)}$ may have only finitely many discontinuities, each confined to a finite collection of patches—say, $N_d$ such patches. Moreover, on the remaining $(P - N_d)$ patches, we have $f \in C^{m+2}$. By invoking Theorem~\ref{thm:quad_convergence_gen}, the total error can therefore be expressed as
		\begin{equation}
			\left|\int_{a}^{b} f(y)dy - \I^{f1}_{n}[f] \right| \leq C\left[ (P-N_d) \frac{h^{\min\{n+1+n_0,m+3\}}}{n^{m+2}}+ N_d \frac{h^{\min\{n+1+n_0,m+2\}}}{n^{m+2}} \right].
		\end{equation}
		Since $P = \frac{b-a}{h}$, and $N_d$ is a fixed constant, we obtain the required result. 
	\end{proof}
	\section{Computational results}\label{sec:numerics}	

In this section, we present numerical examples illustrating the decay of Chebyshev coefficients with respect to the interval length, as discussed in Section~\ref{sec:decay}, together with computational results demonstrating the applications from Section~\ref{sec:appln}. Throughout, the symbol $n$ denotes the number of terms used when computing discrete Chebyshev coefficients, while in the context of quadrature error, $n$ refers to the number of quadrature nodes.  

To assess performance, we use the \textbf{n}umerical \textbf{d}ecay \textbf{r}ate (ndr) of the discrete Chebyshev coefficients and the \textbf{n}umerical \textbf{o}rder of \textbf{c}onvergence (noc) of the quadrature error, defined by
\begin{equation}\label{eq:decay-rate-def}
   \mathrm{ndr} = \frac{\log\!\left(\widetilde c_{k,h}/\widetilde c_{k,h/2}\right)}{\log(2)},
   \qquad
   \mathrm{noc} = \frac{\log\!\left(\varepsilon_{n}/ \varepsilon_{2n}\right)}{\log(2)},
\end{equation}
where $\widetilde c_{k,h}$ and $\widetilde c_{k,h/2}$ are the Chebyshev coefficients of a given function on intervals of length $h$ and $h/2$, respectively, and $\varepsilon_n$ and $\varepsilon_{2n}$ denote the quadrature errors with $n$ and $2n$ grid points respectively. In the following examples, $m$ denotes the regularity of a given function and $k$ denote the coefficient index. 

	\begin{example}(\textbf{Decay of discrete coefficients: $m$ fixed, $k$ varying})\\
		This example demonstrates the decay rates of the discrete Chebyshev coefficients $\widetilde{c}_k^{f1},\widetilde{c}_k^{cc},\widetilde{c}_k^{f2},\widetilde{c}_k^{f3},\widetilde{c}_k^{f4}$ corresponding to the nodes of the quadratures F-I, CC, F-II, F-III, and F-IV, respectively (see Table~\ref{T:quadratures} for definitions). The coefficients $\widetilde{c}_k$ are computed for $1\leq k\leq n-1$, with $n=8$, for the function $f(x) = x^4|x|+e^x$, using the nodes of the respective quadratures. The computations are performed on intervals $[a_p,b_p]=[-\frac{1}{2p},\frac{1}{p}]$ with the interval length $h_p = \frac{3}{2p}$, where $p$ is successively doubled in each evaluation. As per Corollary~\ref{cor:disCheby_Coeff_Tn}, the expected \textbf{t}heoretical \textbf{d}ecay \textbf{r}ate (tdr) is $\min\{k,m+1\}$, where $m$ is the regularity parameter of the given function. Since $f\in X^m[a_p,b_p]$ with $m=4$ for every $p\geq 1$, we have $\mathrm
		{tdr} = \min\{k,5\}$. As illustrated in Figure~\ref{Fig:decay}, the numerical decay rates (ndr) are in agreement with the theoretical predictions. Notice that the decay rates of the discrete coefficients for all quadratures are the same as per the Corollary~\ref{cor:disCheby_Coeff_Tn}.
        \begin{figure}[H]
		\centering
		\includegraphics[width=13.50cm]{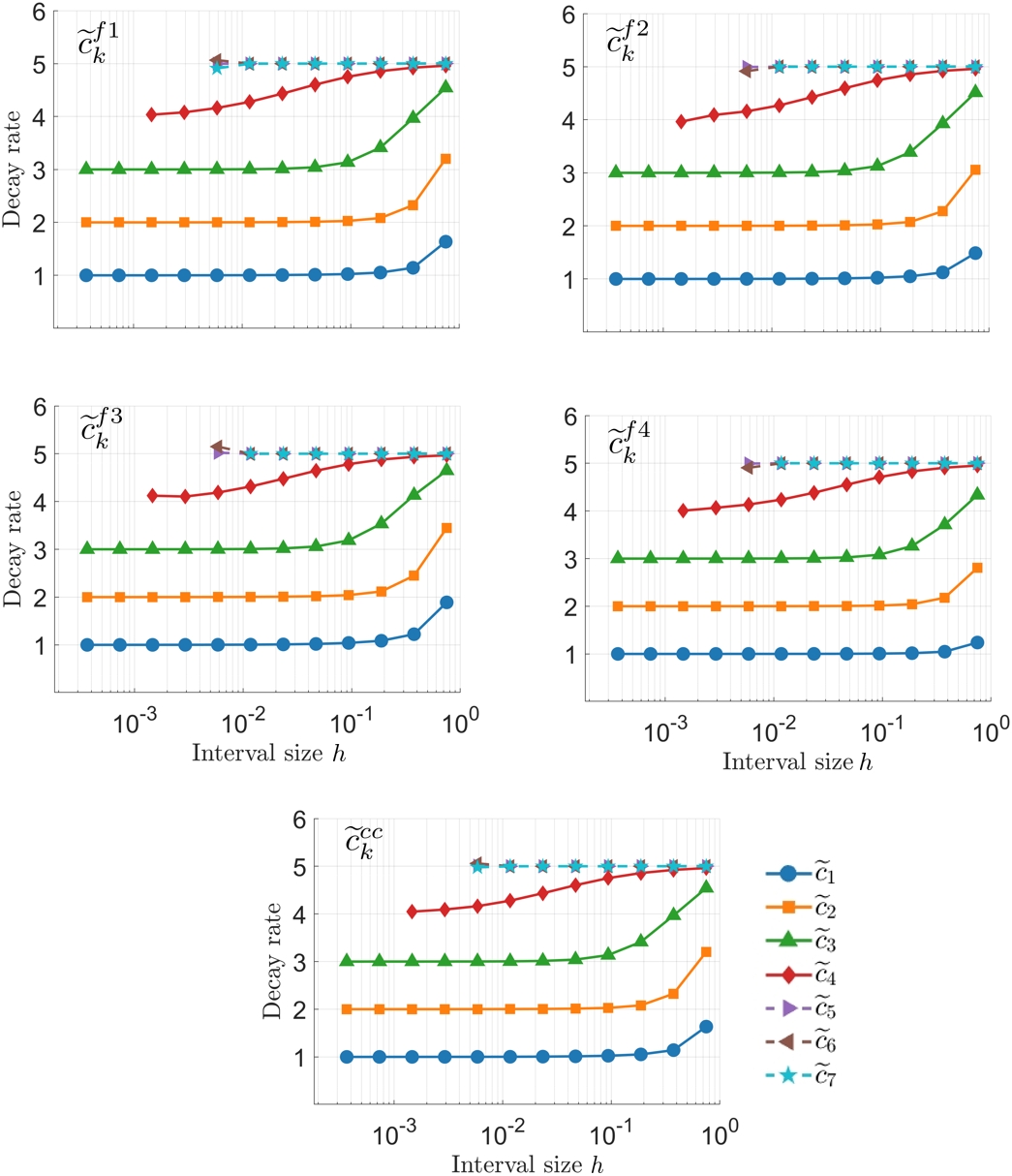}
		\caption{Decay rates of the discrete Chebyshev coefficients $\tilde{c}_k$ for $1\leq k\leq n-1$ with $n=8$ of the function $f(x) = x^4|x|+e^x$ corresponding to the nodes of different quadratures (F-I, F-II, F-III, F-IV, and CC, see Table~\ref{T:quadratures}) on the interval $[a_p,b_p]=[-\frac{1}{2p},\frac{1}{p}]$. \label{Fig:decay}}
	\end{figure}
	\end{example}
	\begin{example}(\textbf{Decay of discrete coefficients: $m$ varying, $k$ fixed})\\
		In this example, we examine the numerical decay rates of the Chebyshev coefficients $\widetilde{c}_k^{f1}$ associated with the nodes of the F-I quadrature. The coefficient index $k$ is fixed, while the regularity parameter $m$ of the function $f(x) = x^m|x|+e^x$, is varied from $0$ to $6$, and the corresponding numerical decay rates are computed. Specifically, graphs in Figure~\ref{Fig:decay1} present the results for $k=3,4,5$. In each case, the observed numerical decay rates (ndr) align precisely with the theoretically predicted rates $\mathrm{tdr}=\min\{k,m+1\}$, yielding $\mathrm{tdr}=\min\{3,m+1\}$, $\mathrm{tdr}=\min\{4,m+1\}$, and $\mathrm{tdr}=\min\{5,m+1\}$, respectively.
        \begin{figure}[H]
		\centering
		\includegraphics[width=15cm]{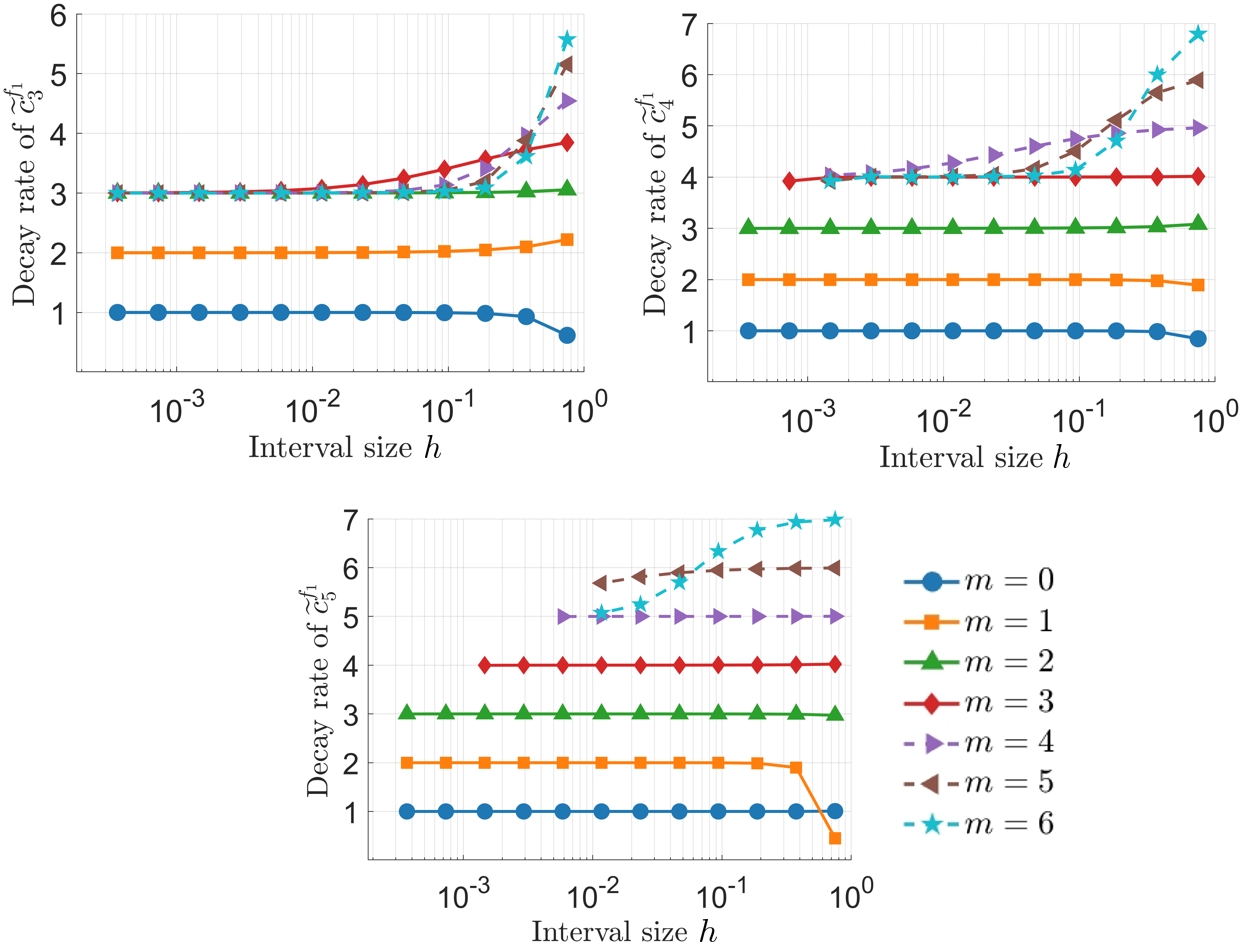}
		\caption{Decay rates of the discrete Chebyshev coefficients $\widetilde{c}_k^{f1}$ for $k=3,4,5$ with $n=8$ for functions $f(x) = x^m|x|+e^x$, $m=0,\dots,6$ corresponding to the nodes of F-I quadrature on the interval $[a_p,b_p]=[-\frac{1}{2p},\frac{1}{p}]$. Similar decay rates are observed for the discrete coefficients corresponding to other quadratures.}
		\label{Fig:decay1}
	\end{figure}
	\end{example}
	\begin{remark}
		Notice that by the triangle inequality~\eqref{eq:triangle} together with Corollary~\ref{cor:disCheby_Coeff_Tn}, the numerical decay rates observed in Figure~\ref{Fig:decay}  and Figure~\ref{Fig:decay1} are intrinsically consistent with both the decay of the continuous Chebyshev coefficients $c_k$ established in Theorem~\ref{Thm_Cheby_Coeff}, and the convergence of the error $|c_k-\widetilde{c}_k|$ as $n \to \infty$ described in Theorem~\ref{thm:error_discrete_and_continuous_ChebyCoeff_Tn}.	
	\end{remark}
	\begin{example}(\textbf{Convergence of F-I quadrature: $m$ varying, $n$ fixed})\\
		In this example, we examine the quadrature errors in approximating the test function $f(x) = x^m|x| + e^x$ over the interval $[a_p,b_p] = [-\frac{1}{2p}, \frac{1}{p}]$, where the interval length is $h_p = \frac{3}{2p}$ and $p$ is successively doubled at each step. The number of quadrature nodes is fixed ($n=8$), while the regularity of the test function is varied by changing the values of $m$ from 0 to 4. The numerical results, reported in Table~\ref{T:FI-n-fix}, show that the observed noc is equal to $m+2$. This is in complete agreement with the theoretical prediction $\text{toc} = \min\{n+1+n_0, m+2\}$ with $n_0=1$ if $n$ is odd and $n_0=0$ if $n$ is even as stated in Theorem~\ref{thm:quad_convergence_gen}, in the framework where the interval length $h_p \to 0$.
	\end{example}
	
	\begin{example}(\textbf{Convergence of F-I quadrature: $m$ fixed, $n$ varying})\\
		This example reports the numerical results of the F-I quadrature for approximating the integral of $f(x) = x^{10}|x| + e^x$ over the interval $[a_p,b_p] = \Bigl[-\tfrac{1}{2p}, \tfrac{1}{p}\Bigr]$ with the number of quadrature points $n$ varying. 
		The expected order of convergence, given by Theorem~\ref{thm:quad_convergence_gen}, is $\mathrm{toc} = \min\{n+1+n_0,\, m+2\}$, 
		where $n_0 = 1$ if $n$ is odd and $n_0 = 0$ if $n$ is even. The theoretical rates are in agreement with the numerical results shown in Table~\ref{T:FI-n-vary}, where we vary the number of quadrature points $n$ from $n=1$ to $n=4$ and the $\mathrm{noc}=n+1+n_0$.
	\end{example}

	\begin{table}[H]
		\centering
		\scalebox{0.910}{\begin{tabular}{c| c| c| c| c|c| c|c| c|c|c}
    \hline \hline
    \multirow{2}{*}{$p$} & 
    \multicolumn{2}{c|}{$m=0$ (toc=2)}& 
    \multicolumn{2}{c|}{$m=1$ (toc=3)}& 
    \multicolumn{2}{c|}{$m=2$ (toc=4)}& 
    \multicolumn{2}{c|}{$m=3$ (toc=5)} &
    \multicolumn{2}{c}{$m=4$ (toc=6)} \\
    \cline{2-11}
    & $\varepsilon_n$ & noc& $\varepsilon_n$ & noc& $\varepsilon_n$ & noc& $\varepsilon_n$ & noc & $\varepsilon_n$ & noc \\ 
    \cline{1-3}
    \hline \hline
    $1$&  $4.80 \times 10^{-3}$&  -&  $5.53 \times 10^{-4}$&  -&  $5.20 \times 10^{-5}$&  -&  $2.19 \times 10^{-5}$&  -& $1.46 \times 10^{-6}$ &-\\
    $2$&  $1.20 \times 10^{-3}$&  $2.00$&  $6.91 \times 10^{-5}$&  $3.00$&  $3.25 \times 10^{-6}$&  $4.00$&  $6.84 \times 10^{-7}$&  $5.00$& $2.28 \times 10^{-8}$ &$6.00$\\
    $4$&  $3.00 \times 10^{-4}$&  $2.00$&  $8.64 \times 10^{-6}$&  $3.00$&  $2.03 \times 10^{-7}$&  $4.00$&  $2.14 \times 10^{-8}$&  $5.00$& $3.56 \times 10^{-10}$ &$6.00$\\
    $8$&  $7.49 \times 10^{-5}$&  $2.00$&  $1.08 \times 10^{-6}$&  $3.00$&  $1.27 \times 10^{-8}$&  $4.00$&  $6.68 \times 10^{-10}$&  $5.00$& $5.57 \times 10^{-12}$ &$6.00$\\
    $16$&  $1.87 \times 10^{-5}$&  $2.00$&  $1.35 \times 10^{-7}$&  $3.00$&  $7.93 \times 10^{-10}$&  $4.00$&  $2.09 \times 10^{-11}$&  $5.00$& $8.70 \times 10^{-14}$ &$6.00$\\
    $32$&  $4.68 \times 10^{-6}$&  $2.00$&  $1.69 \times 10^{-8}$&  $3.00$&  $4.96 \times 10^{-11}$&  $4.00$&  $6.53 \times 10^{-13}$&  $5.00$& $1.30 \times 10^{-15}$ &$6.07$\\
    $64$&  $1.17 \times 10^{-6}$&  $2.00$&  $2.11 \times 10^{-9}$&  $3.00$&  $3.10 \times 10^{-12}$&  $4.00$&  $2.04 \times 10^{-14}$&  $5.00$& $4.16 \times 10^{-17}$ &$4.96$\\
    $128$&  $2.93 \times 10^{-7}$&  $2.00$&  $2.64 \times 10^{-10}$&  $3.00$&  $1.94 \times 10^{-13}$&  $4.00$&  $5.55 \times 10^{-16}$&  $5.20$& - &-\\
    $256$&  $7.32 \times 10^{-8}$&  $2.00$&  $3.30 \times 10^{-11}$&  $3.00$&  $1.20 \times 10^{-14}$&  $4.01$&  -&  -& - &-\\
    $512$&  $1.83 \times 10^{-8}$&  $2.00$&  $4.12 \times 10^{-12}$&  $3.00$&  $7.33 \times 10^{-16}$&  $4.03$&  -&  -& - &-\\
    $1024$& $4.57 \times 10^{-9}$& $2.00$& $5.15 \times 10^{-13}$& $3.00$& -& -& -& -&- &-\\
    \hline \hline
\end{tabular}} 
\caption{Convergence illustration of the Fej\'er first quadrature (with fixed quadrature nodes, specifically $n=8$) for approximating $f(x)=x^m|x|+e^x$ over the interval $[a_p,b_p] = \left[-\tfrac{1}{2p}, \tfrac{1}{p}\right]$, as $(b_p - a_p) \to 0$.    
         }
		\label{T:FI-n-fix}
	\end{table}
	
	\begin{table}[H]
		\centering
		\scalebox{0.89}{
			\begin{tabular}{c| c| c| c| c|c| c|c| c|c|c}
    \hline \hline
    \multirow{2}{*}{$p$} & 
    \multicolumn{2}{c|}{$n=1$ (toc=3)}& 
    \multicolumn{2}{c|}{$n=2$ (toc=3)}& 
    \multicolumn{2}{c|}{$n=3$ (toc=5)}& 
    \multicolumn{2}{c|}{$n=4$ (toc=5)} &
    \multicolumn{2}{c}{$n=5$ (toc=7)} \\
    \cline{2-11}
    & $\varepsilon_n$ & noc& $\varepsilon_n$ & noc& $\varepsilon_n$ & noc& $\varepsilon_n$ & noc & $\varepsilon_n$ & noc \\ 
    \cline{1-3}
    \hline \hline
    $1$&  $2.69 \times 10^{-1}$&  -&  $5.72 \times 10^{-2}$&  -&  $2.19 \times 10^{-2}$&  -&  $2.13 \times 10^{-2}$&  -&  $6.81 \times 10^{-3}$&-\\
    $2$&  $2.01 \times 10^{-2}$&  $3.74$&  $9.99 \times 10^{-3}$&  $2.52$&  $4.02 \times 10^{-5}$&  $9.09$&  $1.11 \times 10^{-5}$&  $10.91$&  $1.67 \times 10^{-6}$&$11.99$\\
    $4$&  $2.34 \times 10^{-3}$&  $3.10$&  $1.17 \times 10^{-3}$&  $3.09$&  $1.03 \times 10^{-6}$&  $5.29$&  $1.73 \times 10^{-7}$&  $6.00$&  $4.78 \times 10^{-10}$&$11.77$\\
    $8$&  $2.84 \times 10^{-4}$&  $3.05$&  $1.42 \times 10^{-4}$&  $3.05$&  $3.11 \times 10^{-8}$&  $5.05$&  $5.19 \times 10^{-9}$&  $5.06$&  $6.42 \times 10^{-13}$&$9.54$\\
    $16$&  $3.49 \times 10^{-5}$&  $3.02$&  $1.74 \times 10^{-5}$&  $3.02$&  $9.58 \times 10^{-10}$&  $5.02$&  $1.60 \times 10^{-10}$&  $5.02$&  $4.27 \times 10^{-15}$&$7.23$\\
    $32$&  $4.33 \times 10^{-6}$&  $3.01$&  $2.16 \times 10^{-6}$&  $3.01$&  $2.97 \times 10^{-11}$&  $5.01$&  $4.95 \times 10^{-12}$&  $5.01$&  $2.78 \times 10^{-17}$&$7.27$\\
    $64$&  $5.39 \times 10^{-7}$&  $3.01$&  $2.69 \times 10^{-7}$&  $3.01$&  $9.24 \times 10^{-13}$&  $5.01$&  $1.54 \times 10^{-13}$&  $5.01$&  -&-\\
    $128$&  $6.72 \times 10^{-8}$&  $3.00$&  $3.36 \times 10^{-8}$&  $3.00$&  $2.87 \times 10^{-14}$&  $5.01$&  $4.72 \times 10^{-15}$&  $5.03$&  -&  - \\
    $256$&  $8.39 \times 10^{-9}$&  $3.00$&  $4.20 \times 10^{-9}$&  $3.00$&  $1.00 \times 10^{-15}$&  $4.85$&  $2.51 \times 10^{-16}$&  $4.24$&  -&  - \\
    $512$& $1.05 \times 10^{-9}$& $3.00$& $5.24 \times 10^{-10}$& $3.00$& $5.12 \times 10^{-17}$& $4.29$& $2.78 \times 10^{-17}$& $3.17$& -&-\\
    $1024$& $1.31 \times 10^{-10}$& $3.00$& $6.55 \times 10^{-11}$& $3.00$& -&   - & &   - & -&  - \\
    \hline \hline
\end{tabular}}
		\caption{Convergence results of the Fej\'er first quadrature, with varying numbers of nodes $n$, for the approximation of the integral of $f(x)=x^{10}|x|+e^x$ as the domain size tends to zero.}
		\label{T:FI-n-vary}
	\end{table}
	\section{Conclusions}\label{conclusion}	
    We established decay estimates for Chebyshev coefficients on shrinking intervals for all four kinds of Chebyshev polynomials and used them to quantify errors in discrete approximations and Chebyshev-based quadratures. The analysis shows how coefficient decay and quadrature accuracy depend on function regularity and interval length. Presented numerical results corroborated these findings: the observed decay of coefficients and quadrature error were consistent with the theoretical estimates. These results are important in designing high-order numerical algorithms that achieve efficiency through localized spectral approximation.
	\section{Acknowledgements}
	Krishna acknowledges the financial support from CSIR through the file no. 09/1020(0183)/2019-EMR-I.
	
	\appendix
    \section{appendix}
    \begin{lemma}\label{lem:decay}
		Let $f$ be a continuous function on a bounded interval $[a,b]$ and let $x_0 = \frac{a+b}{2}$. Then the Chebyshev coefficients $c_k$ of $f$ satisfies the following:
		\begin{equation*}
			c_0\rightarrow f\left(x_0\right),\quad \text{and}\quad c_k \rightarrow 0 \quad \text{for every}\quad k\ge 1 \quad \text {as}\quad h \rightarrow 0.
		\end{equation*}
		
	\end{lemma}  
	\begin{proof}
		The zeroth Chebyshev coefficient of $\widetilde{f}$ can be written as,
		\[ c_0=\frac{1}{\pi} \int_{-1}^1 \widetilde{f}(t) \frac{1}{\sqrt{1-t^2}} dt = \frac{1}{\pi} \int_0^\pi \widetilde{f}(\cos \theta) d \theta, \]
		and for $k \geq 1$,
		\[
		c_k=\frac{2}{\pi} \int_{-1}^1 \widetilde{f}(t) \frac{T_k(t)}{\sqrt{1-t^2}} d t=\frac{2}{\pi} \int_0^\pi \widetilde{f}(\cos \theta) \cos (k \theta) d \theta ,\]
		where $\widetilde{f}$ is as given in \eqref{eq:util_def}.
		Applying the Dominated Convergence Theorem, we obtain  
		\[\lim _{h \rightarrow 0} c_0=\frac{1}{\pi} \int_0^\pi \lim _{h \rightarrow 0} \widetilde{f}(\cos \theta) d \theta=f\left(x_0\right),\]
		and for $k\ge 1$,
		\[\lim _{h \rightarrow 0} c_k=\frac{2}{\pi} \int_0^\pi \lim _{h \rightarrow 0} \widetilde{f}(\cos \theta) \cos (k \theta) d \theta=\frac{2}{\pi}u(x_0) \int_0^\pi \cos (k \theta) d \theta=0 .\]
	\end{proof}
    	\begin{lemma}\label{lemma_Cheby_Coeff_form}
		Let $f \in X^m[a,b]$ and $1\leq M\leq m+1$, $k \in \mathbb{N}$, then
		\begin{equation}\label{eq_Cheby_Coeff_form}
			c_k = \sum_{i=1}^{M} \sum_{j=0}^{i} \frac{h^{i}}{k^M}\alpha_{ij}^{(M)}  I_{M}[f,(i,j,k)],
		\end{equation}
		where $\alpha_{ij}^{(M)}$ are constants defined recursively and independent of $k$ and $h$, that is, for $0\le j \le i$, $1\le i \le M$,
        $\alpha_{ij}^{(M)} = \alpha_{1,ij}^{(M)}+\alpha_{2,ij}^{(M)}+\alpha_{3,ij}^{(M)}$,  where
        \begin{equation*}
		\begin{aligned}
			&\alpha_{1,ij}^{(M)}=\begin{cases}  
				0,  &\text{if}\quad i=M \text{ or } j=i\\
				(-1)^{M} (j+1)\alpha_{i(j+1)}^{(M-1)} , &\text{ otherwise, }\end{cases}\\                
			&\alpha_{2,ij}^{(M)}= \begin{cases}  
				0, &\text{if}\quad i=M \text{ or } j=0\\ 
				(-1)^{M} (i-j+1)\alpha_{i(j-1)}^{(M-1)}  , &\text{ otherwise, }\end{cases}\\
                &\alpha_{3,ij}^{(M)}=\begin{cases}  
				0, &\text{if}\quad i=1 \text{ or } j=0\\ 
				\frac{(-1)^{M-1}}{2}\alpha_{(i-1)(j-1)}^{M-1}  , &\text{ otherwise, }\end{cases}
		\end{aligned}
        \end{equation*}
		and $\alpha_{10}^{(1)}=0,\alpha_{11}^{(1)} = \frac{1}{2\pi}$. Moreover, $I_{m+1}[f,(m+1,j,k)]=O\left(\frac{1}{k}\right)$, and if $f \in C^{m+2}[a,b]$ then \eqref{eq_Cheby_Coeff_form} holds for $1 \le M \le m+2$.
	\end{lemma}

	\begin{proof}
		We prove this lemma using induction on $M$. When $M=1$, the lemma is true since $c_{k} = \frac{1}{2\pi} \frac{h}{k} I_{1}[f,(1,1,k)]$ as obtained from integration by parts in~\eqref{Eq_Cheby_Coeff_intPart_one}. By induction hypothesis, assume that the result is true for any natural number $M\le m$. Using \Cref{lemma_expression_Ck_induction}, the coefficient $c_{k}$ can be expressed as
		\begin{align*}
			c_k &=   \sum_{i=1}^{M} \sum_{j=0}^{i-1} \frac{h^{i}}{k^{M+1}}\left[(-1)^{M+1} \alpha_{i(j+1)}^{(M)} (j+1) \right]I_{M+1}[f,(i,j,k)] \\ 
			&\quad +\sum_{i=1}^{M} \sum_{j=1}^{i} \frac{h^{i}}{k^{M+1}}\left[(-1)^{M} \alpha_{i(j-1)}^{(M)} (i-j+1) \right]I_{M+1}[f,(i,j,k)]\\
			&\quad +\sum_{i=2}^{M+1} \sum_{j=1}^{i} \frac{h^{i}}{k^{M+1}}\left[\frac{(-1)^M}{2}\alpha_{(i-1)(j-1)}^{(M)}\right] I_{M+1}[f,(i,j,k)]\\
			&=  \sum_{i=1}^{M+1} \sum_{j=0}^{i} \frac{h^{i}}{k^{M+1}}\alpha_{ij}^{(M+1)} I_{M+1}[f,(i,j,k)].
		\end{align*}
		Thus, equation \eqref{eq_Cheby_Coeff_form} is true for $1\le M\le m+1$. This completes the proof by induction. Additionally, since $\widetilde{f}(\cos{(\cdot)})\in C^{m+2}_{\mathrm{pw}}[-\pi,\pi]$, there exist discontinuities of $(m+1)$-th derivative of $\widetilde{f}(\cos{\theta})$ at $\{s_\ell\}_{\ell=0}^{n_d}$ for some $n_{d}\in \N$ such that $[-\pi,\pi] = \cup_{\ell=1}^{n_d} [s_{\ell-1},s_\ell]$. Therefore,
		\begin{align*} I_{m+1}[f,(m+1,j,k)] &= \sum_{\ell=1}^{n_d}\int_{s_{\ell-1}}^{s_{\ell}} \widetilde{f}^{(m+1)}(\cos{\theta})\sin^{j}(\theta)\cos^{i-j}(\theta)\zeta_{m+1}(k\theta)d\theta \\ &= \frac{(-1)^{m+1}}{k}\sum_{\ell=1}^{n_d}\left[ \left(\widetilde{f}^{(m+1)}(\cos{\theta})\sin^{j}(\theta)\cos^{i-j}(\theta) \zeta_{m+2}(k\theta)\right)_{s_{\ell-1}}^{s_\ell} \right. \nonumber\\
			&\quad \left. -\int_{s_{\ell-1}}^{s_{\ell}} \left(\widetilde{f}^{(m+1)}(\cos{\theta})\sin^{j}(\theta)\cos^{i-j}(\theta)\right)'\zeta_{m+2}(k\theta)d\theta \right] = O\left(\frac{1}{k}\right).
		\end{align*} 
		This completes the proof for case $f\in X^{m}[a,b]$. If $f\in C^{m+2}[a,b]$ then another application of recursive relation \Cref{lemma_expression_Ck_induction} proves \eqref{eq_Cheby_Coeff_form}  for $M=m+2$.
	\end{proof}

	\section{Appendix}
	The discrete orthogonality relations presented in this section of the appendix are well documented in the literature for indices $i$ and $k$ within the range $0 \leq i, k \leq n-1$ \cite[Ch. 4]{mason2002chebyshev}. However, we have found limited references addressing the case of discrete orthogonality for $i \geq 0$ and $0 \leq k \leq n-1$, particularly across all Chebyshev polynomials and the corresponding nodes summarized in Table~\ref{T:quadratures}. For the convenience of the reader, we provide here a direct proof of discrete orthogonality for the Chebyshev polynomials $T_n$ with respect to the quadrature nodes of Fej\'er first quadrature and state analogous discrete orthogonality results for related cases.  
	
	\begin{lemma}\label{lem:dis_otho_gen}(\textbf{Discrete orthogonality relations})\\
		Let $T_n$, $U_n$, $V_n$, and $W_n$ denote the Chebyshev polynomials of the first, second, third, and fourth kinds, respectively, as defined in~\eqref{eq:ChebyPoly}.
		Let $t_j^{f1}$, $t_j^{f2}$, $t_j^{f3}$, $t_j^{f4}$ and $t_j^{cc}$ represent the quadrature nodes corresponding to the Fejér first, Fejér second, Fejér third, Fejér fourth, and Clenshaw–Curtis quadratures, respectively, as specified in Table~\ref{T:quadratures}.
		Then, for $0 \leq k \leq n-1$ and $i \geq 0$, the discrete orthogonality relations for the Chebyshev polynomial $T_n$, with respect to $ t_j^{f1}$ and $t_j^{cc}$, and Chebyshev polynomials $U_n$, $V_n$, and $W_n$ with respect to the quadrature nodes $t_j^{f2}$, $t_j^{f3}$, and $t_j^{f4}$, respectively, are given by
		\begin{alignat}{2}
			\sum_{j=0}^{n-1} T_i(t_j^{f1})T_k(t_j^{f1}) &= \frac{n}{2}\begin{dcases}
				(-1)^{\frac{i\mp k}{2n}}, & 2n \mid i\mp k, \ (i,k)\ne(0,0) \\
				2, & (i,k)=(0,0)\\
				0, & \text{otherwise},
			\end{dcases}.\label{eq:dis_otho_FFQ}\\
			\sum_{j=0}^{n-1} \frac{1}{\widetilde \gamma_j} T_i\left(t_j^{cc}\right)T_k\left(t_j^{cc}\right) &= \left(\frac{n-1}{2} \right)\begin{dcases}
				\widetilde\gamma_k &, 2(n-1) \mid i-k, \\ 
				1 &,   2(n-1) \mid i+k, \ 0<k<n-1 \\ 
				0 &, \text{otherwise},
			\end{dcases} \label{lem:dis_otho_CC}\\
			\sum_{j=0}^{n-1} U_i\left(t_j^{f2}\right)U_k\left(t_j^{f2}\right)\left(1-(t_j^{f2})^2\right) &= \left(\frac{n+1}{2}\right)\begin{dcases}
				1 &, 2(n+1) \mid i-k, \\
				-1 &, 2(n+1) \mid i+k+2, \\
				0 &, \text{otherwise},
			\end{dcases} \label{lem:dis_otho_Un}\\
			\sum_{j=0}^{n-1} V_i\left(t_j^{f3}\right)V_k\left(t_j^{f3}\right)\left(1+t_j^{f3}\right) &= \left(n+\frac{1}{2}\right)\begin{dcases}
				(-1)^{\frac{i-k}{2n+1}} &, 2n+1 \mid i-k, \\
				(-1)^{\frac{i+k+1}{2n+1}} &, 2n+1 \mid i+k+1, \\
				0 &, \text{otherwise},
			\end{dcases}\label{lem:dis_otho_Vn}\\
			\sum_{j=0}^{n-1} W_i\left(t_j^{f4}\right)W_k\left(t_j^{f4}\right)\left(1-t_j^{f4}\right) &= \left(n+\frac{1}{2}\right)\begin{dcases}
				1 &, 2n+1 \mid i-k, \\
				-1 &, 2n+1 \mid i+k+1, \\
				0 &, \text{otherwise},
			\end{dcases}\label{lem:dis_otho_Wn}
		\end{alignat}
		where $\widetilde{\gamma}_k$ is as defined in \eqref{eq:gamma-tilde}.
	\end{lemma}    
	\begin{proof}
		We provide the proof of~\eqref{eq:dis_otho_FFQ}, the remaining cases follow analogously. Denote the sum on the left-hand side of~\eqref{eq:dis_otho_FFQ} by $S$. Using a trigonometric identity of product of cosines, we may re-write the sum $S$ as 
		\begin{equation}
			S = \sum_{j=0}^{n-1} T_i(t_j^{f1})T_k(t_j^{f1}) =  \sum_{j=0}^{n-1} \cos{(i\theta_j^{f1})}\cos{(k\theta_j^{f1})} = \frac{1}{2}\sum_{j=0}^{n-1} \left( \cos{((i-k)\theta_j^{f1})}+\cos{((i+k)\theta_j^{f1})} \right),
		\end{equation}
		where $\theta_j^{f1}$ is as given in Table \ref{T:quadratures}.
		Introducing the notations $a=\tfrac{(i-k)\pi}{2n}$ and $b=\tfrac{(i+k)\pi}{2n}$, and applying standard formula for summing cosines with arguments in arithmetic progression (see  \cite[Eq. (4.37)-(4.39)]{mason2002chebyshev}), we obtain
		\begin{align}
			S &= \frac{1}{2}\sum_{j=0}^{n-1} \left( \cos{(a+2ja)}+\cos{(b+2jb)}\right) \\
			&= \frac{1}{2}\begin{cases}
				\frac{\sin{(2na)}}{2\sin{(a)}} &,  a\notin \pi \mathbb{Z}\\
				n \cos{(a)} &, a \in \pi \mathbb{Z}
			\end{cases}+ \frac{1}{2}\begin{cases}
				\frac{\sin{(2nb)}}{2\sin{(b)}} &,  b\notin \pi \mathbb{Z}\\
				n \cos{(b)} &, b \in \pi \mathbb{Z}
			\end{cases} \\
			&=  \frac{1}{2}\begin{cases}
				0 &,  2n \nmid i-k\\
				n(-1)^{\frac{i-k}{2n}} &, 2n \mid i-k
			\end{cases}+ \frac{1}{2}\begin{cases}
				0 &,  2n \nmid i+k\\
				n(-1)^{\frac{i+k}{2n}} &, 2n \mid i+k
			\end{cases}. 
		\end{align}
		Next, we consider the divisibility conditions. Suppose $2n \mid (i-k)$ and $(i,k)\neq(0,0)$. Then necessarily $2n \nmid (i+k)$. Indeed, if $2n \mid (i+k)$ as well, then  $2n | ((i+k)-(i-k)) = 2k$,
		which implies $n \mid k$. This is impossible since $0 < k < n$. Hence, in this case we obtain  $S = \frac{n}{2}(-1)^{\frac{i-k}{2n}}$.
		Conversely, by contraposition, if $(i,k)\neq(0,0)$ and $2n \mid (i+k)$, then $2n \nmid (i-k)$, which yields  $S = \frac{n}{2}(-1)^{\frac{i+k}{2n}}$.
		Finally, when $(i,k)=(0,0)$ we have $S=n$, while in all remaining cases we obtain $S=0$. This completes the proof.
	\end{proof}
	
	\begin{lemma}\label{Ap1}(\textbf{Integral of powers of sine and cosine})\\
		For non-negative integers $\ell,~q$ and $M$, the following identity hold: 
		\begin{equation*}
			\int_{-\pi}^{\pi} \sin^{\ell}{(\theta)}\cos^{q}{(\theta)}\zeta_M(k\theta) d\theta = 0,
		\end{equation*}
		where $\zeta_M(\theta) = \cos{\theta}$ if $M$ is even, else $\zeta_M(\theta) = \sin{\theta}$.
	\end{lemma}
	\begin{proof}
It suffices to show that for non-negative integers $\ell,q,k$, such that $\ell + q < k$, \[\int_{-\pi}^{\pi} \sin^\ell (\theta)\cos^q (\theta) e^{-ik\theta} dx = \int_{-\pi}^{\pi} \sin^\ell (\theta)\cos^q (\theta) \cos{(k\theta)} dx-i\int_{-\pi}^{\pi} \sin^\ell (\theta)\cos^q (\theta) \sin{(k\theta)} dx= 0.  \]Substituting Euler representations of $\sin{(\theta)} = \dfrac{e^{i\theta}-e^{-i\theta}}{2i}$ and $\cos{(\theta)} = \dfrac{e^{i\theta}+e^{-i\theta}}{2}$, 
	\begin{gather*}
		\int_{-\pi}^{\pi} \sin^\ell (\theta)\cos^q (\theta) e^{-ik\theta} d\theta  = \int_{-\pi}^{\pi} \left(\frac{e^{i\theta}-e^{-i\theta}}{2i}\right)^\ell \left(\frac{e^{i\theta}+e^{-i\theta}}{2}\right)^q e^{-ik\theta} d\theta
		\\
		=\frac{1}{2\pi}\frac{1}{2^{\ell+q} i^{\ell}}\sum_{r_1=0}^{\ell}\sum_{r_2=0}^{q}\dbinom{\ell}{r_1}\dbinom{q}{r_2}(-1)^{r_1}\int_{-\pi}^{\pi} e^{i(\ell+q-2(r_1+r_2)-k)\theta} d\theta.
	\end{gather*} 
	Since  $\ell+q-2(r_1+r_2)-k \le (\ell+q)-k < 0$, the argument in the exponential power is non-zero. Thus, the above integral is zero. 
	\end{proof}
		
	\bibliographystyle{plain}

	\bibliography{references.bib} 
	
\end{document}